\documentclass[11pt,oneside, a4paper]{amsart}

\usepackage[english]{babel}
\usepackage[T1]{fontenc}
\usepackage{amsmath}
\usepackage{amssymb}
\usepackage{float}
\usepackage{amsfonts}
\usepackage{mathtools}
\usepackage[utf8]{inputenc}
\usepackage[mathcal]{eucal}
\usepackage{enumerate}
\usepackage{amsthm}
\usepackage{hyperref}
\usepackage{bm,upgreek}
\usepackage[totalwidth=14cm,totalheight=20cm]{geometry}
\usepackage{epsfig,fancyhdr,color}

\newtheorem{cor}{Corollary}[section]

\newtheorem{teo}[cor]{Theorem}

\newtheorem{prop}[cor]{Proposition}

\newtheorem{lemma}[cor]{Lemma}

\theoremstyle{definition}
\newtheorem{defi}[cor]{Definition}
\theoremstyle{remark}
\newtheorem{remark}[cor]{Remark}
\newtheorem*{remark*}{Remark}

\newcommand{\R}{\mathbb{R}}

\newcommand{\Z}{\mathbb{Z}}

\newcommand{\h}{\mathbb{H}}

\newcommand{\Ima}{\mathrm{Im}}
\newcommand{\Sec}{\mathrm{Sec}}

\newcommand{\trace}{\mathrm{tr}}
\newcommand{\Tp}{\mathpzc{T}_{\bm{\theta}}(\Sigma_{\mathpzc{p}})}
\newcommand{\Hopf}{\mathrm{Hopf}}
\newcommand{\dTp}{\Tp \times \Tp}
\newcommand{\II}{I\hspace{-0.1cm}I}
\newcommand{\III}{I\hspace{-0.1cm}I\hspace{-0.1cm}I}

\DeclareMathAlphabet{\mathpzc}{OT1}{pzc}{m}{it}
\DeclareMathOperator{\arctanh}{arctanh}

\title[CMC foliation with particles]{Constant mean curvature foliation of globally hyperbolic (2+1)-spacetimes with particles}
\author{Qiyu Chen}
\address{Qiyu Chen:
School of Mathematics, Sun Yat-Sen University,
510275, Guangzhou, P. R. China}
\email{chenqy0121@gmail.com}
\thanks{Q. Chen was partially supported by NSFC, No.11271378.}

\author{Andrea Tamburelli}
\address{Andrea Tamburelli:
University of Luxembourg, Maison du nombre, 6 avenue de la Fonte, L-4364 Esch-sur-Alzette, Luxembourg}
\email{andrea.tamburelli@uni.lu}
\thanks{}
\date{\today}

\begin{document}

\begin{abstract}
Let $M$ be a globally hyperbolic maximal compact
$3$-dimensional spacetime locally modelled on
Minkowski, anti-de Sitter or de Sitter space. It is well known that $M$ admits a unique foliation by constant mean curvature surfaces. In this paper we extend this result to singular spacetimes with particles (cone singularities of angles less than $\pi$ along time-like geodesics).
\end{abstract}

\maketitle

\setcounter{tocdepth}{1}
\tableofcontents

\section*{Introduction}

Globally hyperbolic maximal compact (GHMC) $(2+1)$-dimensional spacetimes with constant curvature have attracted the attention of the mathematical community after the pioneering work of Mess (\cite{Mess}), who pointed out many connections with Teichm\"uller theory. Although the local geometry of these manifolds is trivial, being locally isometric to Minkowski, anti-de Sitter or de Sitter space, the global geometry is more complicated but well-understood at least in the classical, non-singular context. Namely, after the works of Barbot, Béguin and Zeghib (\cite{barbotzeghib}, \cite{BBZFeuilletage}), we have a satisfactory description of their geometry in terms of foliations by constant mean curvature surfaces (called $H$-surfaces henceforth) and constant Gauss curvature surfaces (in brief $K$-surfaces). \\
\indent The aim of this paper is to provide such a description also for 
convex globally hyperbolic maximal (CGHM) $(2+1)$-dimensional spacetimes with particles. These are Lorentzian manifolds locally modelled on Minkowski, anti-de Sitter or de Sitter space, whose metrics present 
cone singularities of angle $\theta_{i} \in (0,\pi)$ along a finite number of time-like geodesics $l_i$. They turn out to be interesting objects to study both from a mathematical and from a physical point of view, being them connected with the Teichm\"uller theory of surfaces with marked points \cite{Schlenker-Krasnov} and modelling the presence of 
point particles in 3d gravity (\cite{thooft1,thooft2}).  \\
\indent This program has already been started by the first author and Jean-Marc Schlenker in \cite{QiyuAdS} and \cite{QiyudS}, where the existence and uniqueness of the foliation by $K$-surfaces is proved for anti-de Sitter and de-Sitter manifolds with particles. In this paper we address the question about the constant mean curvature foliation and we obtain the following result:

\begin{teo}\label{thm:CMCfoliationintro} Let $M$ be a CGHM Minkowski, anti-de Sitter or de-Sitter manifold with particles. Then $M$ admits a unique foliation by constant mean curvature surfaces orthogonal to the singular lines.
\end{teo}

Our proof of this theorem is actually different in the flat and non-flat case. In the anti-de Sitter and de-Sitter setting we exploit the aforementioned results about the existence and uniqueness of the foliation by constant Gauss curvature surfaces orthogonal to the singular lines and deduce Theorem \ref{thm:CMCfoliationintro}, using the fact that every $H$-surface is at constant time-like distance from a $K$-surface, where $H$ can be expressed explicitly as a function of $K$. This is an elementary but new observation and we deduce, in particular, that the existence of a foliation by constant mean curvature surfaces implies the existence of the foliation by constant Gauss curvature surfaces. This is surprising as $H$-surfaces are described by quasi-linear elliptic PDE, whereas the differential equation defining surfaces with constant Gauss curvature is fully non-linear. Moreover, this holds also in the smooth setting, thus obtaining a new proof of the results in \cite{BBZFeuilletage}. \\
\\
\indent On the other hand, in the flat case we prove Theorem \ref{thm:CMCfoliationintro} by introducing a new parametrisation of the deformation space $\mathcal{G}\mathcal{H}^{0}_{\bm{\theta}}(\Sigma_{\mathpzc{p}})$ of CGHM Minkowski manifolds with particles, where $\Sigma_{\mathpzc{p}}$ denotes a closed surface $\Sigma$ of genus $\tau$ with marked points $\mathpzc{p}=(p_{1}, \dots, p_{n})$ such that
\[
\chi(\Sigma,\bm{\theta})=2-2\tau+\sum_{i=1}^{n}\left(\frac{\theta_{i}}{2\pi}-1\right)<0
\]
and $\bm{\theta}=(\theta_{1}, \dots, \theta_{n})\in (0,\pi)^{n}$ represents the cone angles of the cone singularities of 
the CGHM Minkowski  metrics on the manifold $\Sigma_{\mathpzc{p}}\times \R$ along the time-like lines $\{p_{i}\} \times \R$. For every fixed $H\in (-\infty, 0)$, we define a map $\Phi_{H}: T^{*}\Tp \rightarrow \mathcal{G}\mathcal{H}^{0}_{\bm{\theta}}(\Sigma_{\mathpzc{p}})$ which associates to every point $(h,q)\in T^{*}\Tp$, where $h \in \Tp$ is a hyperbolic metric with cone singularities of angle $\theta_i$ at $p_i$, and $q$ is a meromorphic quadratic differential (with respect to the conformal class of $h$) with at most simple poles at the marked points, the isotopy class of the CGHM Minkowski manifold with particles containing an embedded $H$-surface orthogonal to the singular lines with induced metric $I=e^{2u}h$ and second fundamental form 
$\II=\Re(q)+HI$. 
Here the function $u: \Sigma_{p} \rightarrow \R$ is the solution of an elliptic PDE, which translates the fact that the operator $B=I^{-1}\II$ must satisfy the Gauss equation. We then prove the following:

\begin{teo}\label{thm:newparaintro}The map $\Phi_{H}: T^{*}\Tp \rightarrow \mathcal{G}\mathcal{H}^{0}_{\bm{\theta}}(\Sigma_{\mathpzc{p}})$ is a homeomorphism for every $H\in(-\infty,0)$.
\end{teo}

We thus deduce the existence of a unique $H$-surface $S_{H}$ orthogonal to the singular lines in every CGHM Minkowki manifolds with particles. Moreover, by studying the relation between $K$-surfaces and $H$-surfaces in Minkowski manifolds, we obtain that the family $\{S_{H}\}_{H<0}$ provides a foliation of the whole manifold as a consequence of the following result:

\begin{teo}\label{thm:KfoliationMinkintro} Let $M$ be a CGHM Minkowski manifold with particles. Then it admits a unique foliation by constant Gauss curvature surfaces orthogonal to the singular lines.
\end{teo}

We also underline the fact that the existence of surfaces with prescribed Gauss or mean curvature is obtained in this paper without using the classical barrier arguments or exploiting the results of the non-singular case, thus our techniques provide also a new proof of the results in \cite{barbotzeghib}.

\subsection*{Outline of the paper} In Section \ref{background} and Section \ref{models} we recall the main definitions and well-known results about hyperbolic metrics with cone singularities on surfaces and $(2+1)$-dimensional spacetimes with particles. Section \ref{sec:Minkowsky} is dedicated to the proof of Theorem \ref{thm:newparaintro} and to the study of constant Gauss curvature and constant mean curvature foliations of CGHM Minkowski manifolds with particles. Theorem \ref{thm:CMCfoliationintro} is proved for the anti-de Sitter manifolds in Section \ref{sec:AdS} and for de Sitter manifolds in Section \ref{sec:dS}.

\section{Teichm\"uller space of hyperbolic metrics with cone singularities}\label{background}
In this section we introduce hyperbolic metrics with cone singularities and their deformation space. The material covered here is classical, the main purpose being fixing the notation and recalling the main results that we are going to use in the sequel.

\subsection{Hyperbolic metrics with cone singularities}
Let $\Sigma$ be a closed oriented surface of genus $\tau$ with $n$ marked points $p_{1},\dots, p_{n}$.
Let $\mathpzc{p}=(p_{1},\dots, p_{n})$. We will denote by $\Sigma_{\mathpzc{p}}$ the surface $\Sigma$ with the marked points $\mathpzc{p}$ removed. Let us fix an $n$-tuple of angles $\bm{\theta}=(\theta_{1}, \dots, \theta_{n}) \in (0,\pi)^{n}$.

 Let $\theta\in(0,2\pi)$. We first recall the local model of a hyperbolic metric with a cone singularity of angle $\theta$. Let $\mathbb{H}^2$ be the Poincar\'{e} model of the hyperbolic plane.  The local model is constructed in the following way: consider two half-lines in $\mathbb{H}^2$ which intersect at the center $O$ and bound a wedge of angle $\theta$ at $O$, and then glue the two half-lines by a rotation 
  fixing $O$. The resulting space is denoted by $\h_{\theta}$, called the \emph{hyperbolic disk with cone singularity of angle $\theta$}, endowed with the metric that in polar coordinates assumes the form:
\[
	g_{\h_{\theta}}=dr^{2}+\sinh^{2}(r)d\phi^{2},   \ \ \  r\in\R^{+}  \ , \ \phi \in \R/\theta\Z \ .
\]

\begin{defi} \label{Def:hyperbolic metrics with cone singularities}
A metric on $\Sigma_{\mathpzc{p}}$ with cone singularities of angles $\bm{\theta}$ at $\mathpzc{p}$
is a smooth metric $g$ with constant sectional curvature $-1$ on $\Sigma_{\mathpzc{p}}$ such that for every $i=1, \dots, n$, there exists a neighborhood $U_i$ in $\Sigma$ of $p_i$, such that the metric $g_{|_{U_i\setminus \{p_{i}\}}}$ is the pull back by a local diffeomorphism $\psi_i$ (which satisfies a regularity condition near $p_i$) of the metric $g_{\h_{\theta_{i}}}$.
\end{defi}

\begin{remark}
 Note that the hyperbolic metrics near cone singularities throughout this paper satisfy a technical assumption. That is, they satisfy a regularity condition such that the local diffeomorphism $\psi_i$ used in Definition \ref{Def:hyperbolic metrics with cone singularities} belongs to a certain type of weighted H\"{o}lder spaces (see \cite[Section 2.2]{GR} for more details). This ensures the existence of harmonic maps from Riemann surfaces with marked points to hyperbolic surfaces with cone singularities at the marked points (see \cite[Theorem 2]{GR}).
\end{remark}

For a hyperbolic metric with cone singularities of angles $\bm{\theta}$ at $\mathpzc{p}$, there exists a conformal coordinate $z$ in a neighborhood $U_{i}$ of $p_{i}$ such that
\[
	g_{|_{U_{i}\setminus \{p_{i}\}}}=e^{2u_{i}(z)}|z|^{2\beta_{i}}|dz|^{2},
\]
where $u_{i}:U_{i}\rightarrow \R$ is a continuous bounded function and 
$\beta_{i}$ denotes the defect
\[
	\beta_{i}=\frac{\theta_{i}}{2\pi}-1 \in \left(-1,-\frac{1}{2}\right) \ .
\]

\noindent We will always assume from now on that $\tau$ and $\beta_{i}$ satisfy
\[
\chi(\Sigma,\bm{\theta})=2-2\tau+\sum_{i=1}^{n}\beta_{i}<0 \ .
\]
This guarantees (\cite{Tro}) that the surface $\Sigma_{\mathpzc{p}}$ admits a hyperbolic metric with cone singularities of angles $\theta_{i}$ at the marked points $p_{i}$.

\begin{remark}The choice of considering singularities of angles $\theta_{i}<\pi$ ensures that at each marked point $p_{i}$, we can always find a neighborhood of fixed radius (depending only on $\theta_i$) isometric to a neighborhood of the singular point in $\h_{\theta_{i}}$, such that it does not contain any other singularities $p_j$ for $j\not=i$ (in other words, the Collar Lemma still holds near cone singularities, see \cite[Theorem 3]{DP} for more details).
\end{remark}

We denote by $\Tp$ the set of hyperbolic metrics on $\Sigma_{\mathpzc{p}}$ with cone singularities of angles $\bm{\theta}=(\theta_{1},\dots, \theta_{n}) \in (0,\pi)^{n}$ at $\mathpzc{p}=(p_{1},\dots p_{n})$, up to diffeomorphisms of $\Sigma$ isotopic to the identity (by an isotopy that fixes each marked point $p_{i}$).

\subsection{Codazzi operators and the tangent space of $\Tp$}
We recall the identification between the tangent space of $\Tp$ 
and the space of meromorphic quadratic differentials on $\Sigma_{\mathpzc{p}}$ with at most simple poles at the marked points. Meromorphic quadratic differentials on $\Sigma_{\mathpzc{p}}$ are then linked to Codazzi operators. These notions will play a fundamental role in Section \ref{sec:Minkowsky}.\\
\\
\indent A Codazzi operator $b$ on a hyperbolic surface $(\Sigma_{\mathpzc{p}},g)$ with cone singularities is a smooth $g$-self-adjoint operator on the tangent bundle of $\Sigma_{\mathpzc{p}}$ that satisfies the Codazzi equations
\[
	(d^{\nabla^{g}}b)(X,Y):=\nabla^{g}_{X}(b(Y))-\nabla^{g}_{Y}(b(X))-b([X,Y])=0 \ \
\]
for all vector fields  $X$, $Y$ on $\Sigma_{\mathpzc{p}}$, where $\nabla^{g}$ is  the Levi-Civita connection of $g$. Moreover, we require some regularity at the singular points: we assume that
\[	
	\int_{S} \trace({b^{2}}) dA_{g}=\int_{S} \| b\|^{2} dA_{g} < +\infty \ ,
\]
where $A_g$ is the area form of the singular metric $g$, and $\| \cdot \|$ denotes the operator norm with respect to the metric $g$. In this case we say that $b$ is \emph{integrable} and we write $b \in L^{2}(\Sigma_{\mathpzc{p}},g)$. \\
\\
\indent Traceless Codazzi operators on $(\Sigma_{\mathpzc{p}}, g)$ are related to meromorphic quadratic differentials on $\Sigma_{\mathpzc{p}}$ and the integrability condition is linked to the order of the poles. Let us explain this relation in more details. Let $q$ be a meromorphic quadratic differential on $\Sigma_{p}$ with respect to the conformal class of the hyperbolic metric $g$: by this we mean that $q$ is holomorphic on $\Sigma_{\mathpzc{p}}$ and possibly with poles at the marked points $p_i$. We can define the operator $b_{q}$ by the relation
\[
	\Re(q)(X,Y)=g(b_{q}(X),Y) \
\]
for all vector fields $X$, $Y$ on $\Sigma_\mathpzc{p}$.
\begin{prop}\label{prop:estimatesB} The operator $b_{q}$ is traceless and Codazzi, and it is integrable if and only if $q$ has at most simple poles at the singular points. Moreover, $b_{q}$ extends continuously to $0$ at the singular points.
\end{prop}
\begin{proof} Let $z$ be a conformal coordinate centered at the puncture $p_i$. In a neighbourhood $U_i$ of $p_i$ we can write
\[
	q=f(z)dz^{2} \ \ \  \ \  \text{and} \ \ \ \ \  h_{|_{U_i\setminus \{p_i\}}}=e^{2u_i}|z|^{2\beta_i}|dz|^{2} ,
\]
where $u_i:U \rightarrow \R$ is a continuous function and $\beta_i \in (-1,-1/2)$. The operator $b_{q}$ can thus be expressed as
\[	
	b_{q}=h^{-1}\Re(q)=e^{-2u_i}|z|^{-2\beta_i} \left(\begin{matrix}
									\Re(f) & -\Ima(f) \\
									-\Ima(f)  & -\Re(f)
								\end{matrix}\right) \ .
\]
Clearly, $b_{q}$ is traceless and an easy computation shows that the Cauchy-Riemann equations for $f$ on $U_i\setminus \{p_i\}$ are equivalent to the Codazzi condition for $b_{q}$. Moreover,
\[
	\int_{U_i} \trace{(b_{q}^{2})}dA_{g}= \int_{U_i} 2e^{-2u_i}|z|^{-2\beta_i}|f(z)|^{2}dz\wedge d\bar{z}
\]
is finite if and only if $f$ has at most a simple pole at $p_i$. \\
If $q$ has at most simple pole at $p_i$, we deduce that $|f(z)|^{2} \leq C|z|^{-2}$, hence $\| b_{q}\|^{2} \leq C|z|^{-4\beta_i-2}$, which tends to $0$ as $|z|\to 0$ by our assumption on the cone angles.
\end{proof}

Conversely, given a traceless, Codazzi and integrable operator $b$ on $(\Sigma_\mathpzc{p}, g)$, there exists a unique meromorphic quadratic differential $q$ on $\Sigma_\mathpzc{p}$ with at most simple poles at singularities with respect to the conformal class of $g$, such that
\[
	\Re(q)(X,Y)=g(b(X),Y) \
\]
for all vector fields $X$, $Y$ on $\Sigma_\mathpzc{p}$.\\
\\
\indent In this way, we can identify the tangent space of $\Tp$ at $g$ with the space of meromorphic quadratic differentials on $\Sigma_{\mathpzc{p}}$ (with respect to the conformal class of $g$) with at most simple poles at the marked points. The identification goes as follows (see \cite{jeremythesis}) for more details). Let $g_{t}$ be a path of hyperbolic metrics on $\Sigma_{\mathpzc{p}}$ with cone singularities of angles $\bm{\theta}$ at $\mathpzc{p}$. The tangent vector
\[
	\tilde{h}=\frac{d}{dt}_{|_{t=0}}g_{t}
\]
can be decomposed uniquely as
\[
	\tilde{h}=h+\mathcal{L}_{V}g_0
\]
where $\mathcal{L}_{V}$ denotes the Lie derivative of $g_0$ induced by an element $V$ in the Lie algebra  of the group of diffeomorphisms of $\Sigma_{\mathpzc{p}}$ isotopic to the identity, and $h$ is an integrable, divergence-free and traceless bilinear form. This implies that $b=g_0^{-1}h$ is an integrable, Codazzi and traceless operator, hence it is the real part of a meromorphic quadratic differential $q$. Moreover, the integrability condition on $b$ implies that $q$ has at most simple poles at the marked points.

\section{Geometric models}\label{models}

In this section we introduce the local models for $(2+1)$-spacetimes with particles and constant sectional curvature. We will also recall the notion of global hyperbolicity and the theory of space-like embeddings of surfaces.

\subsection{Minkowski space with cone singularities}
We will denote by $\R^{2,1}$ the $3$-dimensional Minkowski space. Given $\theta \in (0,\pi)$, the \emph{Minkowski space with cone singularities of angle $\theta$} along a time-like geodesic is obtained according to the following procedure. Consider two totally geodesic time-like half-planes intersecting at a time-like geodesic $l$ bounding a wedge of angle $\theta$, cut out the wedge and re-glue the two time-like half-planes by a rotation 
fixing the line $l$. An easy computation in polar coordinates shows that the resulting space $\R^{2,1}_{\theta}$ is endowed with the metric
\[
	g_{\R^{2,1}_{\theta}}=-dt^{2}+dr^{2}+r^{2}d\phi^{2} \ \ \ t \in \R \ , r \in \R^{+} \ \phi \in \R/\theta\Z\ ,
\]
which carries cone singularities of angle $\theta$ along the time-like line $l=\{r=0\}$.

\subsection{Anti-de Sitter space with cone singularities}

We first recall the definition of the classical anti-de Sitter space and we then introduce cone singularities along time-like geodesics.\\
\\
\indent Let $\R^{2,2}$ be the real $4$-dimensional vector space endowed with the bilinear form of signature $(2,2)$:
\[
	\langle x,  y \rangle_{2,2}=-x_{0}y_{0}+x_{1}y_{1}+x_{2}y_{2}-x_{3}y_{3} \ .
\]
We define the \emph{anti-de Sitter (AdS) space} as
\[
	\widehat{AdS}_{3}=\{ x \in \R^{2,2} \ | \ \langle x, x\rangle_{2,2}=-1 \} \ .
\]
It is easy to verify that $\widehat{AdS}_{3}$ is a $3$-dimensional manifold, diffeomorphic to a solid torus, and the restriction of the bilinear form $\langle \cdot, \cdot \rangle_{2,2}$ to the tangent space of $\widehat{AdS}_{3}$ induces a Lorentzian structure of constant sectional curvature $-1$. \\
\\
\indent In order to introduce cone singularities, it is more convenient to work in the universal cover. Let us consider here the hyperbolic plane $\h^{2}$ realised as one connected component of the two-sheeted hyperboloid in Minkowski space. That is,
\[
	\mathbb{H}^{2}=\{( x_0, x_1, x_2) \in \R^{2,1} \ | -x_0^2+x_1^2+ x_2^2=-1, x_2>0 \} \ .
\]
The map
\begin{align*}
	F: \h^{2}\times S^{1} &\rightarrow \widehat{AdS}_{3}\\
	(x_{0}, x_{1}, x_{2}, e^{i\theta}) &\mapsto (x_{0}\cos(\theta), x_{1}, x_{2}, x_{0}\sin(\theta))
\end{align*}
is a diffeomorphism, hence $\h^{2} \times S^{1}$ is isometric to the anti-de Sitter space, if endowed with the pull-back metric	
\[
	(F^{*}g_{\widehat{AdS_{3}}})_{(x, e^{i\theta})}=(g_{\h^{2}})_{x}-x_{0}^{2}d\theta^{2} \ .
\]
We easily deduce that the universal cover of anti-de Sitter space can be realised as $\widetilde{AdS}_{3} \cong \h^{2}\times\R$ endowed with the Lorentzian metric:
\[
	(g_{\widetilde{AdS_3}})_{(x,t)}=(g_{\h^{2}})_{x}-x_{0}^{2}dt^{2} \ .
\]
By considering polar coordinates on $\h^{2}$, the universal cover of anti-de Sitter space can be parameterised as
\[
   \widetilde{AdS}_{3}=\{(\rho,\varphi, t) \ | \ \rho \in \R^{+}, \ \varphi \in \R/2\pi\Z, \ t \in \R\}
\]
and the Lorentzian metric can be written as
\[
	g_{\widetilde{AdS_3}}=-\cosh^{2}(\rho)dt^{2}+\sinh^{2}(\rho)d\varphi^{2}+d\rho^{2} \ .
\]
\indent We now introduce cone singularities along time-like lines in this model. Let $\theta>0$. The anti-de Sitter (AdS) space with cone singularities of angles $\theta$ is obtained according to the following procedure: consider two totally geodesic time-like half-planes in $\widetilde{AdS_3}$ intersecting along the time-like geodesic $l=\{\rho=0\}$ and bounding a wedge of angle $\theta$,  cut the wedge along these two half-planes and then re-glue the two sides of the wedge by a rotation fixing $l$. 
We will denote by $AdS_{\theta}$ the resulting $3$-manifold, called the \emph{AdS space with cone singularities of angle $\theta$}. We call the line $l$ the \emph{singular line} in $AdS_{\theta}$ and call $\theta$ the \emph{total angle} around the singular line $l$. It is clear from the construction that a parametrisation of $AdS_{\theta}$ is given by
\[
  AdS_{\theta}=\{(\rho,\varphi,t) \ | \ \rho \in \R^{+}, \ \varphi \in \R/\theta\Z, \ t \in \R\}
\]
and it is endowed with the Lorentzian metric
\[
	g_{AdS_{\theta}}=-\cosh^{2}(\rho)dt^{2}+\sinh^{2}(\rho)d\varphi^{2}+d\rho^{2} \ .
\]
We notice, in particular, that $g_{AdS_{\theta}}$ carries cone singularities of angle $\theta$ along the line $l$ and outside this line the metric is smooth and with constant sectional curvature $-1$.

\subsection{De Sitter space with cone singularities}
Now we introduce the related notion in the de Sitter setting.
\\
\indent Consider the real $4$-dimensional vector space $\mathbb{R}^{3,1}$ endowed with the bilinear form of signature $(3,1)$:
\[
	\langle x,  y \rangle_{3,1}=-x_{0}y_{0}+x_{1}y_{1}+x_{2}y_{2}+x_{3}y_{3} \ .
\]

The \emph{de Sitter (dS) space} is defined as the set:
\[
 dS_{3}=\{ x \in \R^{3,1} \ | \ \langle x, x\rangle_{3,1}=1 \} \ ,
\]
It is a 3-dimensional manifold diffeomorphic to $\mathbb{S}^2\times \mathbb{R}$, where $\mathbb{S}^2$ is a unit sphere, and the restriction of the bilinear form $\langle \cdot,  \cdot \rangle_{3,1}$ to the tangent space of $dS_{3}$ induces a Lorentzian metric of constant sectional curvature $+1$. It is time-orientable and we choose the time orientation for which the curve $t\mapsto(\sinh t, \cosh t, 0,0)$ is future-oriented. The group Isom$_{0}({dS}_{3})$ of time-orientation and orientation preserving isometries of $dS_{3}$ is $SO^{+}(3,1)\cong PSL_2(\mathbb{C})$.

Indeed, the map
	$$ G: \mathbb{S}^{2} \times  \R \rightarrow dS_{3}$$
	$$ (\varphi, \alpha, t) \mapsto (\sinh t, \cosh t\sin\varphi \cos\alpha, \cosh t\sin\varphi \sin\alpha, \cosh t\cos\varphi)$$
is a diffeomorphism, hence $\mathbb{S}^{2} \times  \R$ is isometric to $dS_{3}$, if endowed with the pull-back metric	
\[
	(G^{*}g_{dS_{3}})_{(\varphi, \alpha, t)}= -dt^{2}+\cosh^{2}(t)(d\varphi^{2}+\sin^{2}(\varphi)d\alpha^{2}) \ ,
\]
where $(t,\varphi,\alpha)\in\mathbb{R}\times [0,\pi]\times\mathbb{R}/2\pi\mathbb{Z}$.\\
\\
\indent We now introduce cone singularities along time-like lines in this model. Let $\theta>0$.  We define the \emph{de Sitter space with cone singularities of angle $\theta$} as the space
\begin{equation*}
        dS_{\theta}:=\{(t,\varphi,\alpha)\in\mathbb{R}\times [0,\pi]\times\mathbb{R}/\theta\mathbb{Z}\}
\end{equation*}
     with the metric
\begin{equation*}
      g_{dS_{\theta}}=-dt^{2}+\cosh^{2}(t)(d\varphi^{2}+\sin^{2}(\varphi)d\alpha^{2}).
\end{equation*}
We call $\R\times \{0,\pi\}\times \mathbb{R}/\theta\mathbb{Z}$ the \emph{singular line} in $dS_{\theta_0}$ and call $\theta$ the \emph{total angle} around this singular line. One can check that $dS_{\theta}$ is a Lorentzian manifold of constant sectional curvature $+1$ outside the singular line. Indeed, it is obtained from the spherical surface with two cone singularities of angle $\theta$ by taking a warped product with $\mathbb{R}$.

\subsection{Convex GHM spacetimes with particles}
Let us denote by $X_{\theta}$ any one of the models $AdS_{\theta}$, $dS_{\theta}$ or $\R^{2,1}_{\theta}$ introduced above. We are interested in a special class of manifolds locally modelled on 
$X_{\theta}$.

\indent A \emph{constant curvature cone-manifold} is a connected, Lorentzian $3$-manifold $M$ in which every point $p$ has a neighbourhood isometric to an open subset of $X_{\theta}$ for some $\theta>0$. If $\theta$ can be taken to be equal to $2\pi$, we say that $p$ is a \emph{regular} point, otherwise $\theta$ is uniquely determined and we say that $p$ is a \emph{singular} point. \\
\\
\indent In order to define the notion of global hyperbolicity in this singular context, we need to introduce the concept of orthogonality along the singular locus.

\begin{defi}Let $S\subset X_{\theta}$ be a space-like surface which intersects a singular line $l$ at a point $p$. We say that $S$ is \emph{orthogonal} to $l$ at $p$ if the causal distance $d$  to the totally geodesic plane $P$ orthogonal to the singular line at $p$ satisfies
\[
	\lim_{q \to p, \\ q \in S}\frac{d(q,P)}{d_{S}(p,q)}=0 \ ,
\]
where $d_{S}(p,q)$ is the distance between $p$ and $q$ on the surface $S$.

Similarly, if $S$ is a space-like surface embedded in a constant curvature cone-manifold $M$ which intersects a singular line $l'$ at a point $p'$, we say that $S$ is \emph{orthogonal} to $l'$ at $p'$ if there exists a neighbourhood $U\subset M$ of $p'$ isometric to a neighbourhood of a singular point in $X_{\theta}$ for some $\theta\in(0,\pi)$, such that the restriction of this isometry sends $S\cap U$ to a surface orthogonal to $l'$ in $X_{\theta}$.
\end{defi}

\begin{remark}The orthogonality condition ensures that the induced metric on the surface $S$ carries cone singularities of the same angle as the total angle around the singular line at the intersection with the singular locus (\cite{Schlenker-Krasnov}).
\end{remark}

\begin{defi}\label{def:CGHM} A Convex Globally Hyperbolic Maximal (CGHM) manifold with particles is a constant curvature cone-manifold $M$ homeomorphic to $\Sigma\times \R$ such that the singularities along the time-like lines $\{p_{i}\}\times \R$ have fixed cone angles $\theta_{i}\in(0,\pi)$. It also satisfies the following conditions:
\begin{itemize}
	\item Convex Global Hyperbolicity: $M$ contains a 
strictly future-convex space-like surface (called Cauchy surface) which intersects every inextensible causal curve in exactly one point and which is orthogonal to the singular locus.
	\item Maximality: if $M'$ is another manifold with particles and $\phi:M \rightarrow M'$ is any isometric embedding sending a Cauchy surface to a Cauchy surface, then $\phi$ is a global isometry.
\end{itemize}
\end{defi}

\begin{remark}The convexity condition on the Cauchy surface is a technical assumption that guarantees the existence, in the anti-de Sitter setting, of the non-empty smallest convex subset of $M$ homotopic equivalent to $M$, called the \emph{convex core}. We do not know if every globally hyperbolic maximal AdS manifold always contains a convex Cauchy surface.
\end{remark}

\begin{remark} By Definition \ref{def:CGHM} a CGHM manifold with particles modelled on de Sitter or Minkowski space is naturally future complete.
\end{remark}

\subsection{Surfaces in CGHM manifolds with particles}
The classical theory of immersions of space-like surfaces into pseudo-Riemannian manifold can be adapted to this singular setting. \\
\\
\indent Let $S\subset M$ be a space-like surface embedded in a CGHM 
manifold  with particles. 
Suppose that $S$ is orthogonal to the singular lines. The restriction of the Lorentzian metric on $M$ to the tangent space of $S$ induces a Riemannian metric $I$, called the \emph{first fundamental form}, which carries a cone singularity at the intersection with the singular line $\{p_{i}\}\times \R$  of the same angle $\theta_{i}$. \\
The shape operator $B:TS \rightarrow TS$ is defined as
\[
	B(u)= -\nabla_{u}\eta \ ,
\]
where $\eta$ is the future-directed unit normal vector field on $S$ and $\nabla$ is the Levi-Civita connection of $M$. It turns out that $B$ is self-adjoint with respect to the first fundamental form, hence $B$ is diagonalisable at every point and its eigenvalues are called \emph{principal curvatures} of $S$. The \emph{second fundamental form} is the bilinear form on $TS$ defined by
\[
	\II(u,v)=I(B(u),v) \ .
\]
The \emph{third fundamental form} on $S$ is defined as
\[
	\III(u,v)=I(B(u),B(v)) \ .
\]
The \emph{mean curvature} of $S$ is then defined as
\[
	H=\frac{\trace(B)}{2} \ .
\]

We say that a space-like surface $S\subset M$ orthogonal to the singular lines is \emph{(strictly) future-convex} (resp. \emph{past-convex}) if its future $I^{+}(S)$ (resp. its past $I^{-}(S)$) is (strictly) geodesically convex. With our convention, $S$ is \emph{future-convex} (resp. \emph{part-convex}) if and only if at each regular point of $S$ the principal curvatures are negative (resp. positive).\\
\\
\indent We recall that for space-like surfaces embedded in manifolds locally modelled on 
$X_{\theta}$ the classical Gauss-Codazzi equations assume the following form:
\begin{align*}
	d^{\nabla^{I}}B&=0 \ \ \ \ \ \ \ \ \ \ \ \ \ \ \ \ \ \  \ \ \ \ \ \ \ \ \ \ \text{Codazzi \ equation}\\
	\det(B)&=\Sec(X)-K_{I} \ \ \ \ \ \ \ \ \ \ \ \ \text{Gauss \ equation}
\end{align*}
where $\nabla^{I}$ is the Levi-Civita connection of $I$, $K_{I}$ denotes the Gauss curvature of the induced metric $I$, and $\Sec(X)$ is the sectional curvature of the regular set $X$ of model space $X_{\theta}$.

\begin{remark}Following the convention of \cite{Tro}, we only consider metrics on $\Sigma_{\mathpzc{p}}$, whose Gauss curvature extends to a continuous function at $\mathpzc{p}$.
\end{remark}

We concule this section with an elementary result that will be crucial in the sequel.
\begin{prop}\label{integrability} Let $S \subset M$ be a space-like surface embedded in CGHM manifold with particles. Suppose that $S$ is orthogonal to the singular lines and has constant mean curvature $H$. Then the trace-less part of the shape operator of $S$ is integrable.
\end{prop} 
\begin{proof} Let us write $B=B_{0}+HE$ and denote with $\lambda$ the positive eigenvalue of $B_{0}$. We need to prove that 
\[
	\int_{S} \trace(B_{0}^{2}) \ dA_{I} =2\int_{S} \lambda^{2} \  dA_{I} < +\infty \ .
\]
By the Gauss-Bonnet formula for surfaces with cone singularities (\cite[Proposition 1]{Tro}) and the Gauss equation we have 
\begin{align*}
	2\pi \chi(\Sigma_{\mathpzc{p}})&= \int_{S} K_{I} \ dA_{I} = \int_{S} (\Sec(X)-\det(B)) \ dA_{I} \\
		&=\int_{S} (\Sec(X)-H^{2}+\lambda^{2}) \ dA_{I}
\end{align*}
and the result follows since $H$ and $\Sec(X)$ are constant.
\end{proof}

\section{The Minkowski case}\label{sec:Minkowsky}


Let $\Sigma_{\mathpzc{p}}$ be a closed, oriented surface with fixed $n$ marked points $\mathpzc{p}=(p_{1}, \dots, p_{n})$ removed. We denote by $\mathcal{G}\mathcal{H}^{0}_{\bm{\theta}}(\Sigma_{\mathpzc{p}})$ the deformation space of CGHM Minkowski structures with particles on 
$\Sigma\times \R$ with fixed cone angles $\bm{\theta}=(\theta_{1}, \dots, \theta_{n})$, up to diffeomorphisms on $\Sigma\times \R$ isotopic to the identity fixing each singular line.

By studying strictly convex space-like surfaces embedded in CGHM Minkowski manifolds with particles, Bonsante and Seppi obtained a parametrisation of the deformation space $\mathcal{G}\mathcal{H}^{0}_{\bm{\theta}}(\Sigma_{\mathpzc{p}})$.
\begin{teo}\cite[Corollary G]{BSCodazzi}\label{thm:parameterisationMink} $\mathcal{G}\mathcal{H}^{0}_{\bm{\theta}}(\Sigma_{\mathpzc{p}})$ is parameterised by $T\Tp$.
\end{teo}

The correspondence between $\mathcal{G}\mathcal{H}^{0}_{\bm{\theta}}(\Sigma_{\mathpzc{p}})$ and $T\Tp$ is obtained as follows: by assumption a CGHM Minkowski manifold $M$ with particles contains a strictly-convex space-like surface $S$ orthogonal to the singular locus. The parametrisation above associates to $S$ the couple $(\widetilde{h},\widetilde{q})\in \Tp$, where $\widetilde{h}$ is the third fundamental form of $S$ and $\widetilde{q}$ is the meromorphic quadratic differential with respect to the conformal class of $\widetilde{h}$ such that
\[
	\widetilde{h}^{-1}\Re(\widetilde{q})=B^{-1}-\frac{1}{2}\trace(B^{-1})E
\]
is the traceless part of the inverse of the shape operator $B$ on $S$. Moreover, the quadratic differential $\widetilde{q}$ and the isotopy class of $\widetilde{h}$ do not depend on the particular choice of the surface $S$.

\subsection{Relations between H-surfaces and K-surfaces in CGHM Minkowski manifolds with particles}
In this subsection, we suppose that a CGHM Minkowski manifold with particles contains an embedded space-like surface $S_{H}$ orthogonal to the singular locus with constant mean curvature $H<0$. We will show that $M$ contains also an embedded space-like surface $\Sigma_{K}$ orthogonal to the singular lines equidistant to $S_{H}$ with constant Gauss curvature $K=-4H^{2}$. We also prove that there are interesting relations between the induced metrics on $S_{H}$ and $\Sigma_{K}$, and their third fundamental forms, which will play a fundamental role in Section \ref{sec:newparameterisationMink}.\\
\\
\indent We start with an a-priori estimate on the principal curvatures of $S_{H}$.

\begin{lemma}\label{lm:est_princ_curv}Let $S_{H}$ be an $H$-surface orthogonal to the singular locus embedded in a Minkowski manifold with particles. Then the principal curvatures $\mu$ and $\mu'$ of $S_{H}$ satisfy
\[
	2H<\mu\leq H \ \ \ \ \ \ \ \ \text{and} \ \ \ \ \ \ \ \ H \leq \mu' <0
\]
\end{lemma}
\begin{proof}
By Proposition \ref{prop:estimatesB} and Proposition \ref{integrability}, the principal curvatures of $S_{H}$ coincide and are equal to $H$ at the punctures. By continuity, for every $\epsilon>0$, we can find neighbourhoods $U_{i}$ of the punctures $p_{i}$ for $i=1,\dots n$, such that the principal curvatures on $\partial U_{i}$ are in the interval $[H-\epsilon, H+\epsilon]$. Moreover, by choosing $\epsilon$ small enough, we can suppose that $H-\epsilon>2H$ (note that $H<0$).
We denote by $S'$ the surface obtained from $S_{H}$ by removing the open sets $U_{i}$. By construction, the principal curvatures of $S'$ are smooth functions. Let $B$ be the shape operator of $S'$. We consider $B_{0}=B-HE$, the traceless part of $B$. Since $H$ is constant, the operator $B_{0}$ is Codazzi. Let $e_{1}$ and $e_{2}$ be unit tangent vectors in a orthonormal frame of $S'$ that diagonalises $B_{0}$. Since $B_{0}$ is traceless, the eigenvalues are opposite, and we will denote by $\lambda\geq 0$ the eigenvalue of $e_{1}$. Let $\omega$ be the connection $1$-form of the Levi-Civita connection $\nabla$ for the induced metric on $S'$, defined by the relation
\[
	\nabla_{x}e_{1}=\omega(x)e_{2} \ .
\]
The Codazzi equation for $B_{0}$ can be read as follows,
\[
	\begin{cases} \lambda \omega(e_{1})=d\lambda(e_{2}) \\
				  \lambda \omega(e_{2})=-d\lambda(e_{1})  \ .
	\end{cases}
\]
If we define $\beta=\log(\lambda)$ we obtain
\[
	\begin{cases}  \omega(e_{1})=d\beta(e_{2}) \\
				   \omega(e_{2})=-d\beta(e_{1}) \ .
	\end{cases}
\]
Moreover, if we denote by $K$ the Gauss curvature of $S'$, we have
\[
	-K=d\omega(e_{1}, e_{2})=e_{1}(\omega(e_{2}))-e_{2}(\omega(e_{1}))-\omega([e_{1}, e_{2}])=\Delta \beta \ ,
\]
where $\Delta$ is the Laplacian that is positive at 
local maximum. On the other hand by the Gauss equation,
\[
	-K=\det(B)=\det(B_{0}+HE)=\det(B_{0})+H^{2}=-e^{2\beta}+H^{2} \ .
\]
Since the surface $S'$ is compact, the function $\beta$ assumes maximum at a point $x_{0}$. By the fact that $\Delta \beta(x_{0})\geq0$, we deduce that
\[
	\lambda=e^{\beta}\leq \sqrt{H^{2}}=-H \ .
\]

Since the eigenvalues of $B$ are $\mu'=\lambda+H$ and $\mu=-\lambda+H$, we obtain the claim. Notice, moreover, that if $\lambda$ assumes the value $-H$ at some point, this point must be in the interior of $S'$ (since $\lambda=\mu'-H \leq\epsilon<-H$ on $\partial S'$). Hence, by the Maximum Principle, $\lambda$ must be constant and equal to $-H$, but this is a contradiction, since on the boundaries of $S'$, we have
\[
	2H<H-\epsilon\leq\mu=-\lambda+H \ .
\]
\end{proof}

We deduce in particular that $S_{H}$ is always 
future-convex.
Before proving that there exists a surface equidistant to $S_{H}$, in the past of $S_{H}$, with constant Gauss curvature, let us recall how the embedding data of a space-like surface (orthogonal to the singular lines) behave along the normal flow.

\begin{lemma}\label{lm:principalcurvaturesMink} Let $S\subset M$ be a convex space-like surface embedded into a CGHM Minkowski $3$-manifold. Let $\psi^{t}:S \rightarrow M$ be the flow along the future-directed unit normal vector field $\eta$ on $S$. For every regular point $x \in S$, we have
\begin{enumerate}
	\item the induced metric and the shape operator on $\psi^{t}(S)$ are
\[
	I_{t}=I((E-tB)\cdot, (E-tB)\cdot)	\ \ \ \ \ \ \ \text{and} \ \ \ \ \ \ \
	B_{t}=(E-tB)^{-1}B \ ,
\]
	where $I$ and $B$ are the induced metric and the shape operator on $S$.
	\item the principal curvatures of $\psi^{t}(S)$ at the point $\psi^{t}(x)$ are given by
\[
	\lambda_{t}(\psi^{t}(x))=\frac{\lambda(x)}{1-\lambda(x)t} \ \ \ \ \ \ \ \ \
	\mu_{t}(\psi^{t}(x))=\frac{\mu(x)}{1-\mu(x)t} \ ;
\]
	\item $\psi^{t}$ is an embedding in a neighbourhood of $x$ if $t$ satisfies $\lambda(x)t\neq 1$ and $\mu(x)t\neq 1$, where $\lambda(x)$ and $\mu(x)$ are the principal curvatures of $S$ at $x$.
\end{enumerate}
\end{lemma}
\begin{proof} Points (2) and (3) are clear consequences of the formulas in (1). Let us prove (1). It is sufficient to prove it for $M=\R^{2,1}_{\theta}$. Let $\sigma:S \rightarrow M$ be the embedding. The geodesic orthogonal to $S$ at a point $x=\sigma(y)$ can be written as
\[
	\gamma_{x}(t)=\sigma(y)+t\eta(\sigma(y)) \ .
\]
Then,
\begin{align*}
	I_{t}(v,w)&=\langle (d\gamma_{x}(t)(v),d\gamma_{x}(t)(w))\rangle_{2,1} \\
		&=\langle d\sigma_{x}(v)+td(\eta \circ \sigma)_{x}(v), d\sigma_{x}(w)+td(\eta \circ \sigma)_{x}(w) \rangle_{2,1}\\
		&=\sigma_{x}^{*}g_{\R^{2,1}_{\theta}}(v-tBv,w-tBw)\\
		&=I((E-tB)v, (E-tB)w) \ .
\end{align*}
It is well-known that $\II_{t}=-\frac{1}{2}\frac{dI_{t}}{dt}$, hence
\[
	\II_{t}=-\frac{1}{2}\frac{dI_{t}}{dt}=I(B\cdot, (E-tB)\cdot)=I_{t}((E-tB)^{-1}B\cdot,\cdot)
\]
and
\[
	B_{t}=I_{t}^{-1}\II_{t}=(E-tB)^{-1}B \ .
\]
\end{proof}

\begin{remark}The previous formulas hold true also at the singular points, whenever the principal curvatures at those points are well-defined. This is always the case for $H$-surfaces orthogonal to the particles, since in this case the principal curvatures extend continuously at the punctures, where they assume the value $H$ (Proposition \ref{prop:estimatesB}).
\end{remark}

We remark, in particular, that, since with our convention a future-convex space-like surface $S$ has always negative principal curvatures, the 
equidistant surfaces obtained by moving along the normal flow in the future of $S$ are well-defined for every $t$. Moreover, the principal curvatures increase when moving towards the future direction.

\begin{lemma}\label{lm:existenceKsurface}
Let $S_{H}\subset M$ be a future-convex space-like surface, orthogonal to the singular lines. Then the equidistant surface $\psi^{d(H)}(S_{H})$ at an oriented distance $d(H)$ from $S_{H}$ has constant Gauss curvature
\[
	f(H)=-4H^{2}
\]
if
\[
	d(H)=\frac{1}{2H}<0 \ .
\]
\end{lemma}
\begin{proof}By Lemma \ref{lm:est_princ_curv} and by Lemma \ref{lm:principalcurvaturesMink}, the foliation by equidistant surfaces in the past of $S_{H}$ is well-defined for $t \in [\frac{1}{2H},0]$, hence it is sufficient to show that $I_{d(H)}$ is a metric with cone singularities of constant curvature $-4H^{2}$.
First of all, $I_{d(H)}$ is a metric with cone singularities of the same cone angles as $I$, because $S_{H}$ is orthogonal to the singular lines and the surface $\psi^{d(H)}(S_{H})$ is obtained by moving along the normal flow.\\
By Lemma \ref{lm:principalcurvaturesMink}, the metric $I_{d(H)}$ can be written as
\[
	I_{d(H)}=I((E-1/(2H) B)\cdot, (E-1/(2H) B) \cdot )
		=\frac{1}{4H^{2}}I((2HE-B)\cdot, (2HE-B)\cdot) \ .
\]\\
Since the operator $B$ is Codazzi for $I$, the Levi-Civita connection for the metric $I_{d(H)}$ is $\nabla^{I_{d(H)}}=(2EH-B)^{-1}\nabla^{I}(2EH-B)$ and the Gauss curvature of $I_{d_{H}}$ is
\[
	K_{I_{d(H)}}=4H^{2}\frac{K_{I}}{\det(2HE-B)}
                      =4H^{2}\frac{-\det(B)}{4H^{2}-2H\trace(B)+\det(B)}
		              =-4H^{2} \ .
\]
\end{proof}

We denote by $g$ the hyperbolic metric on $\Sigma_{\mathpzc{p}}$ with cone singularities of angle $\bm{\theta}$ in the conformal class of the induced metric on the surface $\Sigma_{K}$ with constant Gauss curvature $K=-4H^{2}$ discovered above. We want to study the relations 
among the above metric $g$, the induced metric $I$ on $S_{H}$ and their third fundamental form $\III$ (which is the same for $S_{H}$ and $\Sigma_{K}$, since they are equidistant).\\
\\
\indent We start with a probably well-known fact:
\begin{lemma}\label{lm:Gaussharmonic}
Let $I$ and $\III$ be the first and third fundamental form on $S_H$ as above. The identity
\[
	G_{H}: (\Sigma_{\mathpzc{p}},I) \rightarrow (\Sigma_{\mathpzc{p}},\III)
\]
is harmonic with Hopf differential
\[
     \Hopf(G_{H})=HIB_{0}+iJ_IHIB_{0} \ ,
\]
where $B_{0}$ is the traceless part of the shape operator $B$ of $S_{H}$, and $J_I$ is the complex structure compatible with $I$.
\end{lemma}
\begin{proof}By definition of the third fundamental form, we have
\[
	G_{H}^{*}\III=I(B\cdot, B\cdot)=I(\cdot, B^{2}) \ .
\]
If we decompose $B=B_{0}+HE$ we obtain that the traceless part of $B^{2}$ is $2HB_{0}$. Since it is Codazzi and integrable, $2HIB_{0}$ is twice the real part of a meromorphic quadratic differential $\Phi_{G_{H}}$ on $(\Sigma, J_{I})$ with at most simple poles at the singularities (see Proposition \ref{prop:estimatesB}), where $J_I$ is the complex structure compatible with $I$. Then we can write
\[
	G_{H}^{*}\III=e(G_{H})I+\Phi_{G_{H}}+\overline{\Phi}_{G_{H}} \ .
\]
Since $\Phi_{G_{H}}$ is holomorphic on $(\Sigma_{\mathpzc{p}}, J_{I})$, $G_{H}$ is harmonic and $\Phi_{G_{H}}$ coincides with the Hopf differential of the map $G_{H}$. By computation, the Hopf differential $\Phi_{G_{H}}$ is $HIB_{0}+iJ_IHIB_{0}$.
\end{proof}

A similar result is obtained when considering the projection along the normal flow from the $S_{H}$ to the equidistant surface $\Sigma_{K}$ with $K=-4H^2$.

\begin{lemma}\label{lm:projharmonic}
Let $g$ be the hyperbolic metric with cone singularities of angle $\bm{\theta}$ in the conformal class of the induced metric on the surface $\Sigma_{K}$, which is at oriented distant $d(H)$ from $S_H$ of constant curvature $K=-4H^2$. The map induced by the projection $\psi_{H}$
\[
	\psi_{H}:(\Sigma_{\mathpzc{p}}, I) \rightarrow (\Sigma_{\mathpzc{p}},g)
\]
is harmonic with Hopf differential
\[
	\Hopf(\psi_{H})=-HIB_{0}-iJ_IHIB_{0} \ ,
\]
where, again, $B_{0}$ is the traceless part of the shape operator $B$ of $S_{H}$ and $J_I$ is the complex structure compatible with $I$.
\end{lemma}

\begin{proof}By definition of the metric $g$ we have
\[
	\psi_{H}^{*}g=I((2HE-B)\cdot, (2HE-B)\cdot)=I(\cdot, (2HE-B)^{2}) \ .
\]
If we decompose $B=B_{0}+HE$, we get that the traceless part of $(2HE-B)^{2}$ is $-2HB_{0}$. 
Since it is Codazzi and integrable, $-2HIB_{0}$ is twice the real part of a meromorphic quadratic differential $\Phi_{\psi_{H}}$ on $(\Sigma, J_{I})$ with at most simple poles at the singularities (see Proposition \ref{prop:estimatesB}), where $J_I$ is the complex structure compatible with $I$. Then we can write
\[
	\psi_{H}^{*}g=e(\psi_{H})I+\Phi_{\psi_{H}}+\overline{\Phi}_{\psi_{H}} \ .
\]
Since $\Phi_{\psi_{H}}$ is holomorphic on $(\Sigma_{\mathpzc{p}}, J_{I})$, $\psi_{H}$ is harmonic and $\Phi_{\psi_{H}}$ coincides with the Hopf differential of the map $\psi_{H}$. By computation, the Hopf differential $\Phi_{\psi_{H}}$ is $-HIB_{0}-iJ_IHIB_{0}$.
\end{proof}

This implies that the map $\psi_{H} \circ G_{H}^{-1}$ is minimal Lagrangian (see Definition \ref{def: minmal Lagrangian} below) and the conformal class of $I$ is the center.

\begin{defi} \label{def: minmal Lagrangian}
An area-preserving diffeomorphism $m:(\Sigma_{\mathpzc{p}}, h) \rightarrow (\Sigma_{\mathpzc{p}},h')$ between hyperbolic surfaces with cone singularities is \emph{minimal Lagrangian} if there exists a conformal class $c \in \Tp$, called the \emph{center}, 
such that the two harmonic maps (note that they exist uniquely by \cite[Theorem 2]{GR}) isotopic to the identity
\[
	f:(\Sigma_{\mathpzc{p}},c)\rightarrow (\Sigma_{\mathpzc{p}},h) \ \ \ \ \ \text{and} \ \ \ \ \
	f':(\Sigma_{\mathpzc{p}},c) \rightarrow (\Sigma_{\mathpzc{p}},h')
\]
have opposite Hopf differentials and
\[
	m=f'\circ f^{-1} \ .
\]
\end{defi}

Minimal Lagrangian maps between hyperbolic surfaces with cone singularities (of fixed angles leas than $\pi$) have an equivalent description in terms of morphisms between tangent bundles (see e.g. \cite[Proposition 2.12]{QiyuAdS}).

\begin{lemma}\label{lm:minimal Lagrangian}
Let $h$, $h'$ be two hyperbolic metrics with cone singularities of angle $\bm{\theta}$. Then $m:(\Sigma_{\mathpzc{p}},h)\rightarrow(\Sigma_{\mathpzc{p}},h')$ is a minimal Lagrangian map if and only if there exists a bundle morphism $b: T\Sigma_{\mathpzc{p}}\rightarrow T\Sigma_{\mathpzc{p}}$ such that
$b$ is self-adjoint for $h$ with positive eigenvalues, $\det(b)=1$, $d^{\nabla}b=0$, $m^{*}h'=h(b\cdot,b\cdot)$ and both eigenvalues of $b$ tend to 1 at the cone singularities.
\end{lemma}


\subsection{A new parameterisation of $\mathcal{G}\mathcal{H}^{0}_{\bm{\theta}}(\Sigma_{\mathpzc{p}})$}\label{sec:newparameterisationMink}
Inspired by the constructions in \cite{Schlenker-Krasnov}, for every $H \in (-\infty, 0)$,  we consider the map
\[
	\Phi_{H}:T\Tp \rightarrow \mathcal{G}\mathcal{H}^{0}_{\bm{\theta}}(\Sigma_{\mathpzc{p}})
\]
that associates to a point $(h, q) \in T\Tp$ the CGHM Minkowski $3$-manifold with particles that contains a convex $H$-surface, whose induced metric $I$ is conformal to $h$ and whose second fundamental form is $\II=\Re(q)+HI$. The main result of this subsection is the following:

\begin{teo}\label{thm:newparameterisation}The map $\Phi_{H}$ is well-defined and is a homeomorphism for every $H \in (-\infty, 0)$.
\end{teo}

Let us verify first that the map is well-defined. This means that, given a point $(h,q) \in T\Tp$, we need to prove that it is always possible to find a smooth function $u: \Sigma_{\mathpzc{p}}\to \R$ which extends continuously at the punctures such that the couple $(I=e^{2u}h, B=I^{-1}\Re(q)+HE)$ satisfies the Gauss-Codazzi equations for surfaces embedded in Minkowski manifolds. Since $H$ is constant, $B$ is always a Codazzi operator for the Levi-Civita connection of $I$. Therefore, we only need to find $u$ such that the Gauss equation is satisfied. A standard computation shows that
\[
	K_{I}=e^{-2u}(\Delta u+K_{h}) \ ,
\]
thus $u$ is a solution of the elliptic partial differential equation
\begin{equation}\label{eq:PDE}
\begin{split}
	e^{-2u}(\Delta u+K_{h})&=K_{I}=-\det(B)    =-H^{2}-\det(B_{0})\\
	&=-H^{2}-e^{-4u}\det(h^{-1}\Re(q)) \ ,
\end{split}
\end{equation}
where we have denoted by $B_{0}=I^{-1}\Re(q)$ the traceless part of $B$. Notice that, since the function $u$ is only continuous in a neighbourhood of the punctures $\mathpzc{p}$, the above equation must be thought of in the weak sense. \\
In the case where $q=0$, Equation (\ref{eq:PDE}) reduces to
\[
	e^{-2u}(\Delta u +K_{h})=-H^{2} \ ,
\]
which is the standard equation for a metric in the conformal class of $h$ with constant curvature $-H^{2}$, hence it is well-known (see \cite{Tro}) that it has a unique solution. \\
In order to deal with the general case, let us introduce the functional:
\begin{align*}
	F_{H}:L^{2}(\Sigma_{\mathpzc{p}},h) &\rightarrow \R \\
u &\mapsto \frac{1}{2}\int_{\Sigma_{\mathpzc{p}}} (\| \nabla u\|^{2}_{h}+H^{2}e^{2u}-e^{-2u}\det(h^{-1}\Re(q))+2K_{h}u) \ dA_{h} \ .
\end{align*}
A standard computation shows that $u$ is a critical point of the functional $F_{H}$ if and only if $u$ is a solution of Equation (\ref{eq:PDE}). Since $F_{H}$ is the sum of four convex functionals (the third term is convex because $\det(h^{-1}\Re(q))\leq0$ since it is traceless) and the first term is strictly convex, then $F_{H}$ is strictly convex, hence it has at most one critical point, which must be a minimum. The existence of the solution of Equation (\ref{eq:PDE}) is a consequence of the following technical lemma:

\begin{lemma}[Proposition 3.7 \cite{Schlenker-Krasnov}]\label{lm:technicalestimate} If $K_{h} \in (-H^{2},0)$ and $q$ is not indentically zero, then there exists a constant $\epsilon>0$, such that
\[
F_{H}(u)\geq \epsilon \| u\|^2_{L^{2}} \ \ \ \ \ \forall \ u \in L^{2}(\Sigma_{\mathpzc{p}},h) \ .
\]
\end{lemma}

\begin{prop}The map $\Phi_{H}$ is well-defined for every $H<0$.
\end{prop}
\begin{proof} Fix $H<0$. The map $\Phi_{H}$ is well-defined if and only if $F_{H}$ has a continuous minimum point $u$, which is smooth outside the punctures. Let $u_{n}$ be a sequence in $L^{2}(\Sigma_{\mathpzc{p}},h)$ minimizing the functional $F_{H}$. By applying the previous proposition to the couple $(2h/H^{2},q)$, the sequence $u_{n}$ remains in a ball of $L^{2}(\Sigma_{\mathpzc{p}},2h/H^{2})$, thus it converges weakly, up to subsequences to a limit $u$. By construction, $u$ is the minimum of $F$, hence it is a weak solution of Equation (\ref{eq:PDE}). By standard regularity theory, $u$ is smooth on $\Sigma_{\mathpzc{p}}$ and extends continuously on $\mathpzc{p}$. \\
The surface obtained from these embedding data $(I=e^{2u}h,B=I^{-1}\Re(q)+HE)$ has constant mean curvature $H<0$ and it is future-convex by Lemma \ref{lm:est_princ_curv}.
\end{proof}

\begin{prop}\label{prop:continuity}The map $\Phi_{H}$ is continuous.
\end{prop}
\begin{proof}We first notice that a CGHM Minkowski manifold with particles is uniquely determined by the embebbing data of an embedded strictly convex space-like surface orthogonal to the singular locus. Namely, if $I$ and $B$ satisfy the Gauss-Codazzi equations for surfaces embedded in Minkowski manifolds, we can consider the manifold
\[
	(\Sigma_\mathpzc{p} \times (-\delta, \delta), g)
\]
where $g=-dt^{2}+I((E-tB)\cdot, (E-tB)\cdot)$. Since $B$ is uniformly negative, it is possible to find $\delta$ such that the metric $g$ is non-degenerate for every $t \in (-\delta, \delta)$. It is easy to check that $g$ endows $\Sigma_{\mathpzc{p}}\times (-\delta, \delta)$ with a Minkowski structure with particles, where the singular lines are $\{p_{i}\}\times (-\delta, \delta)$. By the uniqueness of the maximal extension (see e.g. \cite[Proposition 6.24]{bbsads} for the AdS case), $\Sigma_{\mathpzc{p}}\times (-\delta, \delta)$ is contained in a unique CGHM Minkowski $3$-manifold $M$ with particles.\\
From this remark, we deduce that, to prove continuity of the map $\Phi_{H}$, it is sufficient to show that if $(h_{n},q_{n}) \in T\Tp$ is a sequence converging to $(h,q)\in T\Tp$, then the embedding data $(I_{n}',B_{n}')$ of the $K$-surfaces $(\Sigma_{K})_n$ embedded into $M_{n}=\Phi_{H}(h_{n},q_{n})$, provided by Lemma \ref{lm:existenceKsurface}, converge to the embedding data $(I',B')$ of the $K$-surface $\Sigma_K$ embedded into $M=\Phi_{H}(h,q)$. Note that here $(\Sigma_{K})_n$ (resp. $\Sigma_{K}$) is the equidistant surface at oriented distance $1/(2H)$ from $(S_{H})_n$ (resp. $S_H$), which is an $H$-surface in $M_n$ (resp. $M$) with induced metric $I_n$ (resp. $I$) conformal to $h_n$ (resp. $h$) and second fundamental form $\II_n=\Re(q_n)+HI_n$ (resp. $\II=\Re(q)+HI$), and $K=-4H^2$. Let us write the induced metrics on $(\Sigma_{K})_n$ and $\Sigma_{K}$ respectively as
\[
	I_{n}'=\frac{1}{4H^{2}}g_{n} \ \ \ \ \ \ \ \ \text{and} \ \ \ \ \ \ \ \ \ I'=\frac{1}{4H^{2}}g,
\]
where $g_{n}$ are hyperbolic metrics on 
$\Sigma_{\mathpzc{p}}$ with cone singularities of angle $\bm{\theta}$. Let us denote by $\III_{n}$ the third fundamental form of the K-surface $(\Sigma_{K})_n$ (which is the same as the third fundamental form of the surface $(S_{H})_n$ since they are equidistant). \\
By Lemma \ref{lm:Gaussharmonic} the identity map
\[
	(\Sigma_{\mathpzc{p}},h_{n}) \rightarrow (\Sigma_{\mathpzc{p}},\III_{n})
\]
is harmonic with Hopf differential $Hq_{n}$ (note that the harmonicity is conformally invariant on the domain). Since $h_{n}$ converges to $h$ and $q_{n}$ converges to $q$, the sequence $\III_{n} \in \Tp$ must converge to the unique hyperbolic metric $\III$ on $\Sigma_{\mathpzc{p}}$ with cone singularities of angle $\bm{\theta}$ such that the identity from $(\Sigma_{\mathpzc{p}},h)$ to $(\Sigma_{\mathpzc{p}},\III)$ is harmonic with Hopf differential 
$Hq$. Notice that $\III$ is uniquely determined by Proposition 5.9 in \cite{QiyuAdS} and coincide exactly with the third fundamental form of the $K$-surface $\Sigma_{K}$ embedded in $M=\Phi_{H}(h,q)$.\\
For the same reason the sequence of hyperbolic metrics $g_{n}$ must converge to unique $g \in \Tp$ such that the identity map
\[
	(\Sigma_{\mathpzc{p}}, h) \rightarrow (\Sigma_{\mathpzc{p}}, g)
\]
is harmonic with Hopf differential $-Hq$ and $I'=\frac{1}{4H^{2}}g$ coincides with the induced metric on the $K$-surface $\Sigma_{K}$ embedded in $M$. Hence, $I'_n$ converges to $I'$.\\
Moreover, the sequence of minimal Lagrangian maps isotopic to the identity
\[
	m_{n}:(\Sigma_{\mathpzc{p}},g_{n}) \rightarrow (\Sigma_{\mathpzc{p}},\III_{n})
\]
converges to the unique minimal Lagrangian map isotopic to the identity
\[
	m:(\Sigma_{\mathpzc{p}},g) \rightarrow (\Sigma_{\mathpzc{p}},\III) \
\]
and the corresponding bundle morphisms $b_{n}$ converge to $b$ (where $b_n$ and $b$ are given by Lemma \ref{lm:minimal Lagrangian}). Since the shape operator, denoted by $B'_n$, of $\Sigma_{K,n}$ can be written as
\[
	B_{n}'=-2Hb_{n} \ ,
\]
we deduce that $B_{n}'$ converges to the shape operator $B'$ of $\Sigma_K$ and the proof is complete.
\end{proof}

The next step towards the proof of Theorem \ref{thm:newparameterisation} is the injectivity of the map $\Phi_{H}$. This will follow from an application of the Maximum Principle and from the behaviour of the principal curvatures of a convex surface in a Minkowski manifold along the normal flow.

\begin{prop}\label{prop:injectivity}The map $\Phi_{H}$ is injective for every $H<0$.
\end{prop}
\begin{proof} Suppose by contradiction that $\Phi_{H}$ is not injective for some $H<0$. Then, we could find a CGHM Minkowski manifold with particles that contains two distinct $H$-surfaces $S_{1}$ and $S_{2}$ orthogonal to the singular lines. Up to changing the role of $S_{1}$ and $S_{2}$, only three mutual positions for $S_{1}$ and $S_{2}$ are possible: they intersect transversely at some point; $S_{2}$ is tangent to $S_{1}$ and lies in the future of $S_{1}$; $S_{2}$ is disjoint from $S_{1}$ in the future of $S_{1}$. In any case, since equidistant surfaces from $S_{1}$ provide a well-defined foliation of the future of $S_{1}$, there exists a time $t_{0}> 0$ such that the surface $\psi^{t_{0}}(S_{1})$ is tangent to $S_{2}$ and lies in the future of $S_{2}$. Let $p$ be the point of tangency. We need to distinguish two cases:
\begin{itemize}
	\item if $p$ is a regular point and we denote by $H_{t_{0}}$ the mean curvature of $\psi^{t_{0}}(S_{1})$ we have
\[
	H< H_{t_{0}}(p) \leq  H \ ,
\]
where the first inequality follows by the fact that the mean curvature increases mean moving towards the future (Lemma \ref{lm:principalcurvaturesMink}) and the second inequality follows by the Maximum Principle (\cite{BBZAdSCMC});
	\item if $p$ is a singular point, since all the surfaces considered here are umbilical at the intersections with the singular lines (see Proposition \ref{prop:estimatesB}), the formulas of Lemma \ref{lm:principalcurvaturesMink} hold and the Maximum Principle can be applied (\cite[Proposition 3.6]{QiyuAdS}), hence we obtain again
\[
       H < H_{t_{0}}(p) \leq H \ .
\]
\end{itemize}
In both cases we obtain a contradiction, hence every CGHM Minkowski manifold with particles contains at most one $H$-surface and the injectivity of $\Phi_{H}$ follows.
\end{proof}

In particular, we deduce the following:
\begin{cor}\label{cor:uniqueHfoliationMink} Let $M$ be a CGHM Minkowski manifold with particles. If $M$ admits a foliation by constant mean curvature surfaces orthogonal to the singular lines, then it is unique.
\end{cor}

\subsubsection{Properness of the map $\Phi_{H}$.} By the Invariance of Domain Theorem, the proof of Theorem \ref{thm:newparameterisation} will follow from Proposition \ref{prop:continuity} and Proposition \ref{prop:injectivity} if we show that the map $\Phi_{H}$ is also proper. \\
\\
The key technical lemma towards the proof of the properness of $\Phi_{H}$ is the following:
\begin{lemma}\label{lm:keylemmaproperness}Let $M_{n}=\Phi_{H}(h_{n},q_{n})$ be a convergent sequence of CGHM Minkowski manifolds with particles. Let $g_{n}$ be the sequence of hyperbolic metrics in the conformal class of the induced metrics on the surfaces $(\Sigma_{f(H)})_{n} \subset M_{n}$ with constant Gauss curvature $f(H)=-4H^{2}$, obtained from Lemma \ref{lm:existenceKsurface}. If the metrics $g_{n}$ are contained in a compact set of 
$\Tp$, then the sequence $(h_{n},q_{n})$ converges up to subsequences.
\end{lemma}
\begin{proof}By Theorem \ref{thm:parameterisationMink}, the third fundamental forms $\III_{n}\in \Tp$ of the surfaces $(\Sigma_{f(H)})_{n}$ are contained in a compact subset of $\Tp$. It is proved in Lemma 3.19 of \cite{QiyuAdS}, that if $g_{n}$ and $\III_{n}$ converge in $\Tp$, then the minimal Lagrangian diffeomorphisms $m_{n}:(\Sigma_{\mathpzc{p}}, g_{n}) \rightarrow (\Sigma_{\mathpzc{p}}, \III_{n})$ isotopic to the identity converge to the minimal Lagrangian map $m:(\Sigma_{\mathpzc{p}}, g) \rightarrow (\Sigma_{\mathpzc{p}}, \III)$ isotopic to the identity. In addition, the sequence of centers $c_{n}$ converges to the center $c$, and the Hopf differentials of the harmonic maps involved in the definition (see Definition \ref{def: minmal Lagrangian}) of $m_{n}$ and $m$ converge.

Now, the centers $c_{n}$ are exactly the conformal classes of the induced metrics $I_{n}$ on the surfaces
$(S_{H})_n$ with constant mean curvature $H$ by Lemma \ref{lm:Gaussharmonic} and Lemma \ref{lm:projharmonic}, hence $h_{n}$ converge in $\Tp$. Moreover, by Lemma \ref{lm:Gaussharmonic}, the Hopf differential of the harmonic map from $(\Sigma_{\mathpzc{p}},c_{n})$ to $(\Sigma_{\mathpzc{p}}, \III_{n})$ is, up to a multiplicative constant, exactly $q_{n}$. Therefore, the sequence $q_{n}$ converges.
\end{proof}

The following constructions will help us prove that the sequence $g_{n}$ is contained in a compact set. We first give a description of the positions of two convex space-like surfaces embedded in a CGHM Minkowski manifold with particles, which is an adaption of Lemma 5.7 in \cite{QiyudS} for de Sitter manifolds and Proposition 5.4 in \cite{TambuAdS} for anti-de Sitter manifolds.

\begin{lemma}\label{lm:mutualposition}Let $S_{1}$ and $S_{2}$ be two convex space-like surfaces embedded in a CGHM Minkowski manifold with particles, orthogonal to the singular lines. Suppose that $S_{1}$ and $S_{2}$ are umbilical at the intersection with the singular locus. Then, if the infimum of the Gauss curvature of $S_{1}$ is strictly bigger than the supremum of the Gauss curvature of $S_{2}$, then $S_{1}$ is in the future of $S_{2}$.
\end{lemma}
\begin{proof}Suppose by contradiction that $S_{1}$ is not contained entirely in the future of $S_{2}$. Since the foliation by equidistant surfaces is well-defined in the future of $S_{1}$, there exists a time $t\geq 0$ such that $S_{1}^{t}:=\psi^{t}(S_{1})$ is tangent to $S_{2}$ at a point $p$ and $S_{1}$ lies in the future of $S_{2}$. We distinguish two cases:
\begin{itemize}
	\item if $p$ is a singular point, from the classical Maximum Principle (\cite{BBZAdSCMC}) we deduce that
\[
	K_{S_{2}}(p)\geq K_{S_{1}^{t}}(p) \ .
\]
	Moreover, by Lemma \ref{lm:principalcurvaturesMink} the Gauss curvature increases when moving towards the future, thus we obtain
\[
\sup_{p \in S_{2}}K_{S_{2}}(p) \geq K_{S_{2}}(p)\geq K_{S_{1}^{t}}(p)\geq K_{S_{1}}(p) \geq \inf_{p \in S_{1}} K_{S_{1}}(p)
\]
which contradicts our assumption.
	\item if $p$ is a singular point, since $S_{1}$ is umbilical at the singular points, also $S_{1}^{t}$ is, hence we can apply the Maximum Principle for convex umbilical surfaces (\cite[Lemma 3.12]{QiyuAdS}) and conclude by the same argument as above.
\end{itemize}
\end{proof}

We want to construct a strictly future-convex space-like surface $S$ embedded in a CGHM Minkowski manifold with particles $M$ starting from the data $(\widetilde{h},\widetilde{q})$ provided by Theorem \ref{thm:parameterisationMink} with a precise control on the Gauss curvature of $S$. For every $L<0$, we define the following objects
\begin{align*}
	\widetilde{b}_{0}&=\widetilde{h}^{-1}\Re(\widetilde{q}) \ \ \ \ \ \ \ \ \ \ \ \widetilde{b}=\widetilde{b}_{0}+LE\\
	B&=\widetilde{b}^{-1}\ \ \ \ \ \ \ \ \ \ \ \ \ \ \ \ \ I=\widetilde{h}(\widetilde{b},\widetilde{b}) \ .
\end{align*}

\begin{prop}\label{prop:constructionsurfaces}The following facts hold.
\begin{enumerate}
	\item $I$ is a metric that carries cone singularities of angles $\bm{\theta}$ at $\mathpzc{p}$;
	\item for every constant $C>0$, we can choose $L<0$ such that $B$ is negative definite and $\det(\widetilde{b})>\frac{1}{C}$.
	\item the couple $(I,B)$ satisfies the Gauss-Codazzi equations for convex space-like surfaces embedded in Minkowski manifolds;
	\item the CGHM Minkowski manifold with particles that contains a surface $S$ with embedding data $(I,B)$ is exactly $M$;
	\item $S$ is umbilical at the intersection with the singular lines;
	\item the Gauss curvature of $S$ satisfies $-C<K_{I}<0$.
\end{enumerate}
\end{prop}
\begin{proof} Point (1) is exactly Theorem E in \cite{BSCodazzi}.
\begin{enumerate}\setcounter{enumi}{1}
	\item A simple computation shows that
\[
	\det(\widetilde{b})=\det(\widetilde{b}_{0})+L^{2}.
\]
	Hence, it is sufficient to choose
\[
	L<-\sqrt{\frac{1}{C}-\inf\det(\widetilde{b}_{0})} \ .
\]
	Notice that the square root is always well-defined because, since $\widetilde{b}_{0}$ is traceless, it has negative determinant. Moreover, $B$ is negative definite, as
\[
	\trace(B)=\frac{\trace(\widetilde{b})}{\det(\widetilde{b})}=\frac{2L}{\det(\widetilde{b})}<0, \quad \det(B)=\frac{1}{\det(\widetilde{b})}>0\ .
\]
	\item Since $M$ is constant, $\widetilde{b}$ is Codazzi for $\widetilde{h}$. We deduce that the Levi-Civita connection of the metric $I$ is $\nabla^{I}=\widetilde{b}^{-1}\nabla^{\widetilde{h}}\widetilde{b}$ (where $\nabla^{\widetilde{h}}$ is the Levi-Civita connection of $\widetilde{h}$) and
\begin{align*}
	(d^{\nabla^{I}}B)(u,v)&=\nabla^{I}_{u}(B(v))-\nabla^{I}_{v}(B(u))-B([u,v]) \\
	&=\widetilde{b}^{-1}\nabla^{\widetilde{h}}_{u}v-\widetilde{b}^{-1}\nabla^{\widetilde{h}}_{v}u-\widetilde{b}^{-1}([u,v])\\
	&=\widetilde{b}^{-1}(\nabla^{\widetilde{h}}_{u}v-\nabla^{\widetilde{h}}_{v}u-[u,v])=0 \ .
\end{align*}
Moreover, the Gauss curvature of $I$ satisfies
\[	
	K_{I}=\frac{K_{\bar{h}}}{\det(\bar{b})}=-\det(B) \ .
\]
	\item Since the operator $B$ is negative definite, there exists $\delta>0$ such that the metric
\[
	g=-dt^{2}+I((E-tB)\cdot ,(E-tB)\cdot)
\]
	is well-defined for $t \in (-\delta, \delta)$. It is not difficult to check that $M_{\delta}=(\Sigma\times (-\delta, \delta) , g)$ is a Minkowski manifold with particles and the level $\{t=0\}$ is a space-like surface orthogonal to the singular locus with embedding data $(I,B)$. For the existence and uniqueness of the maximal extension, $M_{\delta}$ is contained in a unique CGHM Minkowski manifold with particles $M'$. In order to check that $M'=M$, it is sufficient to show that the data provided by Theorem \ref{thm:parameterisationMink} for $M$ and $M'$ coincide. We know that the data corresponding to $M$ are $(\widetilde{h}, \widetilde{q}) \in T\Tp$. The data for $M'$ are provided by any convex space-like surface in $M'$ orthogonal to the particles. If we choose $S$, we know that its third fundamental form is
\[
	\III=I(B \cdot, B \cdot)=I(\widetilde{b}^{-1}\cdot, \widetilde{b}^{-1}\cdot)=\widetilde{h}
\]
and the traceless part of the inverse of $B$ is $\widetilde{b}_{0}$ by construction, hence the corresponding quadratic differential is again $\widetilde{q}$. 	
	\item We need to check that at the singular points the eigenvalues of $B$ extend continuously and coincide. By Proposition \ref{prop:estimatesB}, $\widetilde{b}$ extends continuously at the singular points and the eigenvalues both tend to $L$. Since the eigenvalues of $B$ are the inverse of the eigenvalues of $B$, the claim follows.
	\item By the Gauss equation and point (2) and (3) we have
\[
	-C<\frac{-1}{\det(\widetilde{b})}=\frac{K_{\widetilde{h}}}{\det(\widetilde{b})}=K_{I}<0 \ .
\]
\end{enumerate}
\end{proof}

We have now all the ingredients to prove the main result of this paragraph:
\begin{prop}The map $\Phi_{H}$ is proper.
\end{prop}
\begin{proof} We need to prove that if $M_{n}=\Phi_{H}(h_{n},q_{n})$ converges, then also $(h_{n}, q_{n}) \in T\Tp$ converges up to subsequences. By Theorem \ref{thm:parameterisationMink}, we know that the corresponding data $(\widetilde{h}_{n}, \widetilde{q}_{n}) \in T\Tp$ converge. By construction, each manifold $M_{n}$ contains a convex space-like surface $(S_{H})_{n}$ orthogonal to the singular locus with constant mean curvature $H$ and a convex space-like surface $(\Sigma_{f(H)})_{n}$ orthogonal to the singular locus with constant Gauss curvature $f(H)=-4H^{2}$. Let us denote by $g_{n}$ the hyperbolic metric in the conformal class of the induced metric on $(\Sigma_{f(H)})_{n}$. By Lemma \ref{lm:keylemmaproperness}, it is sufficient to show that the sequence $g_{n}$ is contained in a compact set of $\Tp$. \\
For every $L<0$ we define the objects $(\widetilde{b}_{0})_n$, $\widetilde{b}_{n}$, $I_{n}$ and $\III_{n}$ as above. Since the sequences $\widetilde{q}_{n}$ and $\widetilde{h}_{n}$ converge, there exists a constant $C'>0$ such that
\[
	-C'<\det((\widetilde{b}_{0})_n)\leq 0 \
\]
for all $n \in \mathbb{N}$. Fix $\epsilon>0$ such that $4H^{2}-\epsilon>0$. We choose $L<0$ so that
\[
	L<-\sqrt{\frac{1}{4H^{2}-\epsilon}+C'} \ .
\]
With this choice, the surfaces $S_{n}\subset M_{n}$ with the embedding data $(I_{n}, B_{n})$ (see Proposition \ref{prop:constructionsurfaces}) have Gauss curvature uniformly bounded from below:
\[
	-4H^{2}<-4H^{2}+\epsilon\leq K_{I_{n}}<0 \ .
\]
By Lemma \ref{lm:mutualposition}, the surfaces $S_{n} \subset M_{n}$ are contained in the future of the surfaces $(\Sigma_{f(H)})_{n}$. Since the metrics induced on future convex space-like surfaces increase when moving towards the future (\cite[Proposition 4.2]{mehdi}), we deduce that
\[
	\frac{1}{4H^{2}} g_{n} \leq I_{n} \ .
\]
Now, if the sequence $g_{n}$ were not contained in a compact set of $\Tp$, there would exist a simple closed geodesic $\gamma$ such that $\ell_{g_{n}}(\gamma)\to \infty$. But this not possible, since the sequence $\ell_{g_{n}}(\gamma)$ is bounded by $\ell_{I_{n}}(\gamma)$ and the sequence of metrics $I_{n}$ is clearly convergent by construction.
\end{proof}

\subsection{Foliations in CGHM Minkowsky manifolds with particles} A direct consequence of Theorem \ref{thm:newparameterisation} is the following:
\begin{cor}
Let $M$ be a CGHM Minkowski manifold with particles. Then for every $H<0$, there exists a convex space-like surface $S_{H}$ orthogonal to the singular lines with constant mean curvature $H$.
\end{cor}

We deduce then from Lemma \ref{lm:existenceKsurface} also the existence of surfaces with constant Gauss curvature.
\begin{cor}Let $M$ be a CGHM Minkowski manifold with particles. Then for every $K<0$, there exists a convex space-like surface $\Sigma_{K}$ orthogonal to the singular lines with constant Gauss curvature $K$.
\end{cor}

The aim of this section is proving that these two families $\{S_{H}\}_{H<0}$ and $\{\Sigma_{K}\}_{K<0}$ provide a foliation of the whole manifold $M$. Let us start with the family of surfaces with constant Gauss curvature.
\begin{teo}\label{thm:KfoliationMink}
Every CGHM Minkowski manifold with particles $M$ has a unique foliation by constant Gauss curvature surfaces $\{\Sigma_{K}\}_{K<0}$ orthogonal to the singular lines.
\end{teo}
\begin{proof}
We already know that in $M$ we can find a family of constant Gauss curvature surfaces $\Sigma_{K}$ for $K<0$. Let us first prove uniqueness. Suppose that $M$ contains two convex space-like surfaces $\Sigma_{K}$ and $\Sigma_{K}'$ orthogonal to the singular lines with constant Gauss curvature $K$. Up to changing the role of $\Sigma_{K}$ and $\Sigma_{K}'$ only three mutual position are possible: $\Sigma_{K}$ is disjoint from $\Sigma_{K}'$ in the future of $\Sigma_{K}'$; $\Sigma_{K}$ is tangent to $\Sigma_{K}'$ and lies in the future of $\Sigma_{K}'$; or $\Sigma_{K}$ and $\Sigma_{K}'$ intersect transversely. In any case, since the surfaces equidistant to $\Sigma_{K}'$ in the future provide a well-defined foliation of the future of $\Sigma_{K}'$, there exists a time $t>0$ such that the surface $\psi^{t}(\Sigma_{K}')$ obtained by pushing $\Sigma_{K}'$ along the normal flow is tangent to $\Sigma_{K}$ at a point $p$ and lies in the future of $\Sigma_{K}$. We distinguish two cases:
\begin{itemize}
	\item if $p$ is a regular point, by the Maximum Principle and Lemma \ref{lm:principalcurvaturesMink} we deduce that
\[
	K \geq \ K_{\psi^{t}(\Sigma_{K})}(p) > K
\]
	and we obtain a contradiction.
	\item If $p$ is a singular point, since $K$-surfaces are umbilical at the intersection with the singular lines, we can apply again Lemma \ref{lm:principalcurvaturesMink} and the Maximum Principle for umbilical surfaces (\cite[Lemma 3.12]{QiyuAdS}) and conclude as above.
\end{itemize}
We are left to prove that the family $\{\Sigma_{K}\}_{K<0}$ foliates the manifold. Let us first prove that the union of $\Sigma_{K}$ over all $K<0$ foliates a domain in $M$, called $\Omega$. To this aim, it is sufficient to show that if $K_{n}$ converges to $K$, then $\Sigma_{K_{n}}$ converges to $\Sigma_{K}$ uniformly. We will actually prove more: we will show that the convergence is also smooth outside the singular lines. Without loss of generality we can suppose that $K_{n}$ is decreasing and converging to $K$. This implies that the surface $\Sigma_{K_{n}}$ lies in the future of $\Sigma_{K_{n+1}}$ and they all lie in the future of the surface $\Sigma_{K'}$ for some $K'<K$. This kind of sequence of convex space-like surfaces has already been studied in \cite{barbotzeghib}. In particular, the authors proved (Theorem 3.6 and Corollary 3.8 in \cite{barbotzeghib}) that the boundary of the closure of the union of $I^{+}(\Sigma_{K_{n}})$ over all $n$ is a convex space-like surface $\Sigma_{\infty}$. Let us denote by $\sigma_{\infty}:\Sigma_{\mathpzc{p}} \rightarrow M$ the embedding such that $\sigma_{\infty}(\Sigma_{\mathpzc{p}})=\Sigma_{\infty}$. Let $\psi_{n}:\Sigma_{\infty}\rightarrow \Sigma_{K_{n}}$ be the homeomorphisms obtained by following the normal flow. We define $\sigma_{n}=\psi_{n}\circ \sigma_{\infty}$. We claim that $\sigma_{n}$ converges to $\sigma_{\infty}$ in the $C^{\infty}$ topology outside the singular locus and that $\Sigma_{\infty}$ has constant Gauss curvature $K$. The induced metrics on $\Sigma_{n}$ can be written as
\[
	g_{n}=\frac{1}{|K_{n}|}h_{n}
\]
where $h_{n}$ is a hyperbolic metric with cone singularities at $\mathpzc{p}$ of angles $\bm{\theta}$, since all surfaces $\Sigma_{K_{n}}$ are orthogonal to the singular locus. Since the metrics induced on convex space-like surfaces embedded in Minkowski manifold increase when moving towards the future (\cite[Proposition 4.2]{mehdi}) we have that
\[
	\ell_{g_{n}}((\sigma_{n}(\gamma))\leq \ell_{g_{1}}(\sigma_{1}(\gamma))   \ ,
\]
for every simple closed curve $\gamma$ on $\Sigma_{\mathpzc{p}}$. Therefore, we have
\[
	\ell_{h_{n}}(\sigma_{n}(\gamma))=\sqrt{|K_{n}|}\ell_{g_{n}}(\sigma_{n}(\gamma)) \leq
    \sqrt{\frac{|K_{n}|}{|K_{1}|}}\ell_{h_{1}}(\sigma_{1}(\gamma)) \ ,
\]
which implies that the metrics $\sigma_{n}^{*}h_{n}$ are contained in a compact set of $\Tp$. Thus, up to extracting a subsequence, the metrics $\sigma_{n}^{*}h_{n}$ converges to a hyperbolic metric $h_{\infty}\in \Tp$ smoothly outside the singular locus. Following again \cite{barbotzeghib}, for each regular point $x \in \Sigma_{\infty}$ there exists a sequence $x_{n}\in \Sigma_{K_{n}}$, such that $x_{n}$ converges to $x$ and the unit normal of $\Sigma_{K_{n}}$ at $x_{n}$ converges to the unit normal of $\Sigma_{\infty}$ at $x$. This implies that the sequence of $1$-jets of $\sigma_{n}$ converges. Now, from the proof of Theorem 5.6 in \cite{schlenkersurfconv} it follows that if $\sigma_{n}$ were not convergent to $\sigma_{\infty}$ in the $C^{\infty}$ topology, the surface $\Sigma_{\infty}$ would be pleated along a geodesic $\gamma$ (which may be a geodesic segment between singular points) and tangent to a light-like plane somewhere outside $\gamma$. This implies that $\Sigma_{\infty}$ admits a light-like unit normal, which contradicts the fact that $\Sigma_{\infty}$ is space-like. Therefore, $\sigma_{n}$ converges to $\sigma_{\infty}$ in the $C^{\infty}$ topology outside the singular locus and the induced metric on $\Sigma_{\infty}$ is
\[
	g_{\infty}=\frac{1}{|K|}(\sigma_{\infty})_{*}h_{\infty} \ ,
\]
which has constant curvature $K$. Then $\Sigma_{\infty}$ must coincide with $\Sigma_{K}$ by uniqueness.	\\
We now claim that $\Omega=M$. It is sufficient to show that if $K_{n}\to -\infty$, then the union of $I^{+}(\Sigma_{K_{n}})$ is the whole manifold and if $K_{n}'\to 0$ then the union of $I^{-}(\Sigma_{K_{n}'})$ is exactly $M$. We will do the proof for the first case, the other being analogous. If the union $D$ of $I^{+}(\Sigma_{K_{n}})$ does not coincide with $M$, then the boundary $\partial D$ is a convex space-like surface. In particular, the area of $\Sigma_{K_{n}}$ cannot tend to $0$. On the other hand, the Gauss-Bonnet formula for surfaces with cone singularities (\cite[Proposition 1]{Tro}) shows that
\[
   Area(\Sigma_{K_{n}})=\frac{2\pi}{K_{n}}\chi(\Sigma,\bm{\theta}) \ ,
\]
and this gives a contradiction.\\
Moreover, by Lemma \ref{lm:mutualposition}, the surfaces $\Sigma_{K}$ and $\Sigma_{K'}$ are disjoint from each other for all $K\not=K'\in(-\infty,0)$. This completes the proof.
\end{proof}

We deduce that also the family $\{S_{H}\}_{H<0}$ must foliate the whole manifold.
\begin{teo}\label{prop:HfoliationMink}
Every CGHM Minkowski manifold with particles $M$ has a unique foliation by constant mean curvature surfaces $\{S_{H}\}_{H<0}$ orthogonal to the singular lines.
\end{teo}
\begin{proof}Uniqueness was proved in Corollary \ref{cor:uniqueHfoliationMink}.\\
We first prove that the union of $S_{H}$ over $H<0$ foliates a domain in $M$, say $\Omega$. To this aim it is sufficient to show that if $H_{n}$ is a sequence converging to $H \in (-\infty, 0)$, then the surface $S_{H_{n}}$ converges to $S_{H}$. Without loss of generality we can suppose that $H_{n}$ is decreasing and converges to $H$. Let $\Sigma_{K_{n}}$ and $\Sigma_{K}$ be the space-like surfaces, orthogonal to the singular lines, with constant Gauss curvature $K_{n}$ and $K$ respectively, from which the surfaces $S_{H_{n}}$ and $S_{H}$ are constructed. This means that $S_{H_{n}}$ is in the future of $\Sigma_{K_{n}}$ at signed time-like distance $d(H_{n})=1/2H_{n}$. Let us denote by $\sigma_{n}:\Sigma_{\mathpzc{p}} \rightarrow M$ the embedding such that $\sigma_{n}(\Sigma_{\mathpzc{p}})=\Sigma_{K_{n}}$ and similarly $\sigma_{\infty}:\Sigma_{\mathpzc{p}}\rightarrow M$ such that $\sigma_{\infty}(\Sigma_{\mathpzc{p}})=\Sigma_{K}$, as described in the proof of Theorem \ref{thm:KfoliationMink}. By Theorem \ref{thm:KfoliationMink}, we know that $\sigma_{n}$ converges to $\sigma_{\infty}$ uniformly everywhere and smoothly outside the singular locus. Now, by construction the following map $f_{n}:\Sigma_{\mathpzc{p}} \rightarrow M$ is an embedding such that $f_{n}(\Sigma_{\mathpzc{p}})=S_{H_{n}}$:
\[
	f_{n}=\psi^{d(H_{n})}\circ \sigma_{n} \ ,
\]
where $\psi^{t}$ denotes the translation along the normal flow. Since $\sigma_{n}$ converges to $\sigma_{\infty}$ uniformly and the surfaces are future-convex, then for every $q \in \Sigma_{\mathpzc{p}}$, the unit normal of $\Sigma_{K_{n}}$ at $\sigma_{n}(q)$ converges to the unit normal of $\Sigma_{K}$ at $\sigma_{\infty}(q)$ (if $q$ is a singular point, the unit normal is the direction of the singular line). This, together with the continuity of the function $d$, implies that $\psi^{d(H_{n})}\circ\sigma_{n}$ converges to $\psi^{d(H)}\circ\sigma_{\infty}$ uniformly, and this proves the claim. \\
We are now left to prove that $\Omega=M$. Let us suppose by contradiction that
\[
	D^{-}=\bigcup_{H<0}\overline{I^{-}(S_{H})} \subsetneq M \ .
\]
Then there exists a point $p \in M\setminus D^{-}$. Since the family of $K$-surfaces $\Sigma_{K}$ foliates $M$, there exists $K<0$ such that $p \in \Sigma_{K}$. Now, by definition of $\Sigma_{K}$, the surface $S_{H}$ with constant mean curvature $H=f^{-1}(K)$ lies in the future of $\Sigma_{K}$, hence $p \in I^{-}(S_{H})$ and we get a contradiction. \\
Viceversa, if
\[
	D^{+}=\bigcup_{H<0}\overline{I^{+}(S_{H})} \subsetneq M ,
\]
the surface $S^{+}=\partial D^{+}$ would be a convex space-like surface embedded in $M$ (see \cite{barbotzeghib}). By studying the functions $f$ and $d$ defined in Lemma \ref{lm:existenceKsurface}, it is easy to see that for every $\epsilon>0$ there exist $K_{\epsilon}<0$ such that for every $K<K_{\epsilon}$, the constant Gauss curvature surface $\Sigma_{K}$ is at time-like distance less than $\epsilon$ from the corresponding constant mean curvature surface $S_{f^{-1}(K)}$. Let us choose $\epsilon>0$ such that there exists a point $p \in M \setminus D^{+}$ at time-like distance bigger than $2\epsilon$ from $S^{+}$ that lies on a surface $\Sigma_{K}$ for $K<K_{\epsilon}$. The existence of such $\epsilon$ is guaranteed by the fact that the family $\{\Sigma_{K}\}_{K<0}$ foliates the whole manifold and $\Sigma_{K}$ is in the past of $\Sigma_{K'}$ if $K<K'$. This gives the contradiction: the surface $\Sigma_{H}$ corresponding to the surface $\Sigma_{K}$, on which $p$ lies, is not entirely contained in $D^{+}$ by construction.\\
Moreover, note that $S_H$ can be constructed from the constant Gauss curvature surface $\Sigma_{f(H)}$ as the equidistant surface at signed time-like distance $d(f^{-1}(K))$ in the future. Combined with the uniqueness of $S_H$ for each $H<0$ (see Corollary \ref{cor:uniqueHfoliationMink}) and the fact that $\Sigma_{K}$ and $\Sigma_{K'}$ are disjoint from each other for all $K\not=K'\in(-\infty,0)$, it follows that the surfaces $S_{H}$ and $S_{H'}$ are disjoint from each other for all $H\not=H'\in(-\infty,0)$. This completes the proof.
\end{proof}

\section{The anti-de Sitter case}\label{sec:AdS}





This section is dedicated to the proof of the following result:

\begin{teo}\label{mainteoAdS}Let $M$ be a CGHM AdS manifold with particles. For every $H\in (-\infty, +\infty)$, there exists a unique space-like surface $S_{H}$ embedded in $M$, orthogonal to the singular lines and with constant mean curvature $H$. Moreover, the surfaces $\{S_{H}\}_{H \in \R}$ provide a foliation of the manifold $M$.
\end{teo}

The starting point of the proof is the following recent result proved by the first author in a joint work with Jean-Marc Schlenker:
\begin{teo}[\cite{QiyuAdS}]\label{KfoliationAdS} Let $M$ be a CGHM AdS manifold with particles. For every $K\in (-\infty, -1)$ there exists a unique past-convex space-like surface $\Sigma_{K}$ orthogonal to the singular lines and with constant Gauss curvature $K$. Moreover, the surfaces $\{ \Sigma_{K}\}_{K \in (-\infty, -1)}$ provide a $C^2$-foliation of the future of the convex core of $M$ outside the singular locus.
\end{teo}

The main idea of the proof of Theorem \ref{mainteoAdS} is based on the observation that there exist two functions $d:(-\infty,-1) \rightarrow \R^{+}$ and $f:(-\infty, -1) \rightarrow \R$ such that, given a $K$-surface $\Sigma_{K}$, a surface parallel to $\Sigma_{K}$, in the past of $\Sigma_{K}$ at time-like distance $d(K)$ has constant mean curvature $f(K)$. This will imply that we can find a surface $S_{H}$ of constant mean curvature $H$ for every $H \in \R$ and then we deduce that the family $\{S_{H}\}_{H\in \R}$ must foliate the whole manifold $M$ by Theorem \ref{KfoliationAdS}.\\
\\
\indent We first recall the following well-known result about the behaviour of the principal curvatures along the normal flow:
\begin{lemma}[Lemma 3.11 \cite{QiyuAdS}]\label{principalcurvaturesAdS} Let $S\subset M$ be a past-convex space-like surface embedded into a CGHM AdS manifold with particles. Let $\psi^{t}:S \rightarrow M$ be the flow along the future-directed unit normal vector field $\eta$ on $S$. For every regular point $x \in S$, we have
\begin{enumerate}
	\item $\psi^{t}$ is an embedding in a neighbourhood of $x$ if $t$ satisfies $\lambda(x)\tan(t)\neq 1$ and $\mu(x)\tan(t)\neq 1$, where $\lambda(x)$ and $\mu(x)$ are the principal curvatures of $S$ at $x$;
	\item the principal curvatures of $\psi^{t}(S)$ at the point $\psi^{t}(x)$ are given by
\[
	\lambda_{t}(\psi^{t}(x))=\frac{\lambda(x)+\tan(t)}{1-\lambda(x)\tan(t)} \ \ \ \ \ \ \ \ \
	\mu_{t}(\psi^{t}(x))=\frac{\mu(x)+\tan(t)}{1-\mu(x)\tan(t)} \ .
\]
\end{enumerate}
\end{lemma}

\begin{remark}\label{rk: singular points}
The previous formula hold true also at the singular points, whenever the principal curvatures at those points are well-defined.
\end{remark}

We underline in particular that, since with our convention a past-convex space-like surface $S$ has always positive principal curvatures, the equidistant surfaces $\psi^{t}(S)$ obtained by moving along the normal flow on the past of $S$ are well-defined for every $t<0$. Moreover, the principal curvatures of the equidistant surface decrease when moving towards the past.

\begin{lemma}\label{keylemmaAdS}Let $\Sigma_{K}$ be a past-convex space-like surface with constant Gauss curvature $K$, orthogonal to the singular lines. Then the surface $\psi^{-d(K)}(\Sigma_{K})$ in the past of $\Sigma_{K}$ has constant mean curvature
\[
	f(K)=\frac{-2-K}{2\sqrt{-1-K}}
\]
if
\[
	d(K)=\arctan\left(\sqrt{\frac{1}{-1-K}}\right)
\]
\end{lemma}
\begin{proof}Let us denote by $\lambda$ and $\mu$ the principal curvatures of the surface $\Sigma_{K}$. By assumption (and by the Gauss equation) the product $\kappa=\lambda\mu$ is constant. Let us denote by $\lambda_{t}$ and $\mu_{t}$ the principal curvatures of the surface $\psi^{t}(\Sigma_{K})$. We need to show that there exists a time $t<0$, such that the mean curvature $H_{t}=\frac{\mu_{t}+\lambda_{t}}{2}$ is constant on $\psi^t(\Sigma_K)$.
To prove this, it suffices to find $t<0$ such that $dH_{t}=0$. By the previous lemma, we have
\begin{align*}
	&2dH_{t}=(d\lambda_{t}+d\mu_{t})\\
	      &=\left( \frac{d\lambda}{1-\lambda\tan(t)}+\frac{(\lambda+\tan(t))\tan(t)d\lambda}{(1-\lambda\tan(t))^{2}}+ \frac{d\mu}{1-\mu\tan(t)}+\frac{(\mu+\tan(t))\tan(t)d\mu}{(1-\mu\tan(t))^{2}}\right)
\end{align*}
This vanishes identically if and only if
\begin{equation}\label{eq:differential of H_t}
\begin{split}
0&=(d\mu+\tan(t)^{2}d\mu)(1-\lambda\tan(t))^{2}+(d\lambda+\tan^{2}(t)d\lambda)(1-\mu\tan(t))^{2}\\ &=(1+\tan^{2}(t))[d\mu+d\mu\lambda^{2}\tan^{2}(t)+d\lambda+\mu^{2}d\lambda\tan^{2}(t)]\\
&-2(1+\tan^{2}(t))\tan(t)(\lambda d\mu+\mu d\lambda) \ .
\end{split}
\end{equation}

By substituting $\mu=\frac{\kappa}{\lambda}$ in \eqref{eq:differential of H_t} and using the fact that $0=d(\kappa)=d(\lambda\mu)=\mu d\lambda+\lambda d\mu$, we get
\begin{align*}
	0&=\frac{1}{\lambda^{2}}(1+\tan^{2}(t))(-\kappa d\lambda+\lambda^{2}d\lambda-\kappa\tan^{2}(t)\lambda^{2}d\lambda+\kappa^{2}\tan^{2}(t)d\lambda)\\
	&=\frac{d\lambda}{\lambda^{2}}(1+\tan^{2}(t))(\lambda^{2}-\kappa)(1-\kappa\tan^{2}(t))
\end{align*}
and the claim follows from the Gauss equation $\kappa=-1-K$.
The formula for the mean curvature follows then by a direct computation using the previous lemma:
\begin{equation*}
\begin{split}
	2H_{-d(K)}&=\lambda_{-d(K)}+\mu_{-d(K)}=\frac{\lambda-\sqrt{\frac{1}{\lambda\mu}}}{1+\sqrt{\frac{\lambda}{\mu}}}+\frac{\mu-\sqrt{\frac{1}{\lambda\mu}}}{1+\sqrt{\frac{\mu}{\lambda}}}\\
	&=\frac{\lambda+2\sqrt{\lambda\mu}-2\sqrt{\frac{1}{\lambda\mu}}-\frac{1}{\lambda}+\mu-\frac{1}{\mu}}{2+\sqrt{\frac{\lambda}{\mu}}+\sqrt{\frac{\mu}{\lambda}}}=\frac{\lambda+\mu-\frac{1}{\mu}-\frac{1}{\lambda}+2\sqrt{\kappa}-\frac{2}{\sqrt{\kappa}}}{2+\frac{\sqrt{\kappa}}{\mu}+\frac{\sqrt{\kappa}}{\lambda}}\\
	&=\frac{\kappa^{2}+\mu^{2}\kappa-\mu^{2}-\kappa+2\kappa\sqrt{\kappa}\mu-2\sqrt{\kappa}\mu}{2\kappa\mu+\kappa\sqrt{\kappa}+\mu^{2}\sqrt{\kappa}}\\
	&=\frac{(\kappa-1)(\mu^{2}+\kappa+2\sqrt{\kappa}\mu)}{\sqrt{\kappa}(\mu^{2}+\kappa+2\sqrt{\kappa}\mu)} =\frac{\kappa-1}{\sqrt{\kappa}}=\frac{-2-K}{\sqrt{-1-K}}\ ,
\end{split}
\end{equation*}
where we used many times the fact that $\lambda\mu=\kappa$ and in the last equality we used the Gauss formula $\kappa=-1-K$.
\end{proof}

Therefore, we can define two continuous functions
\[
	d:(-\infty,-1)\rightarrow \R^{+} \ \ \ \ \ \ \text{and} \ \ \ \ \ \ f:(-\infty, -1) \rightarrow \R
\]
with the property that, given a past-convex space-like surface $\Sigma_{K}$ with constant Gauss curvature $K$ and orthogonal to the singular locus, the equidistant surface $S_{H}$ in the past of $\Sigma_{K}$ at  time-like distance $d(K)$ has constant mean curvature $H=f(K)$. By studying the image of the function $f$ we easily deduce (using Theorem \ref{KfoliationAdS}) the following:

\begin{cor}Given a CGHM AdS manifold with particles $M$, for every $H \in \R$ there exists a space-like surface $S_{H}$ orthogonal to the singular locus, with constant mean curvature $H$.
\end{cor}

Notice that the surface $S_{H}$ obtained in this way is umbilical, i.e. the principal curvatures extend continuously and coincide, at the singular points. In fact, the $K$-surface $\Sigma_{K}$, provided by Theorem \ref{KfoliationAdS}, is umbilical at the singular points (Proposition 3.6 \cite{QiyuAdS}) and $S_{H}$ is obtained by pushing $\Sigma_{K}$ along the normal flow, hence the claim follows from Lemma \ref{principalcurvaturesAdS} and Remark \ref{rk: singular points}. This property is actually true for every space-like surface $S$, orthogonal to the singular lines and with constant mean curvature, as the following lemma shows. This will be crucial in order to apply the Maximum Principle in this singular context.

\begin{lemma}\label{lm:AdSmeancurvature} Let $S$ be a space-like surface, orthogonal to the singular locus, with constant mean curvature $H$, embedded in a CGHM AdS manifold with particles. Then the principal curvatures of $S$ at the singular points coincide and are equal to $H$.
\end{lemma}
\begin{proof}If $B$ is the shape operator of $S$ and we decompose $B=B_{0}+HE$, where $B_{0}$ is traceless, and Codazzi for $I$ (since $H$ is constant, and $B$ is Codazzi for $I$). By Proposition \ref{prop:estimatesB} and Proposition \ref{integrability}, the eigenvalues of $B_{0}$ tend to $0$ at the singular points and the claim follows.
\end{proof}

Let $S_H$ denote the surface in a CGHM AdS manifold with particles $M$ obtained from the past-convex space-like surface $\Sigma_K$ with constant curvature $K<-1$ with $f(K)=H$ as constructed in Lemma \ref{keylemmaAdS}. We have the following proposition.

\begin{prop}\label{existencefoliationAdS}The surfaces $\{S_{H}\}_{H \in \R}$ foliate the whole manifold $M$.
\end{prop}
\begin{proof}
Let $\Omega$ be the union of the constant mean curvature surfaces $S_{H}$ over all $H\in\R$. An argument similar to that of Proposition \ref{prop:HfoliationMink} shows that if $H_{n}$ is a sequence converging to $H \in (-\infty, +\infty)$, then the surface $S_{H_{n}}$ converges to $S_{H}$ uniformly everywhere and smoothly outside the singular locus.\\
\indent We now prove that the region $\Omega$ must coincide with the whole manifold $M$. Suppose that the family $\{S_{H}\}_{H \in \R}$ does not foliate completely the past of the maximal surface of $M$ (the maximal surface exists uniquely, see \cite[Theorem 1.4]{Jeremy}). Let us denote by $C(M)$ the convex core of $M$. By assumption and the above result, there exists $\delta>0$ such that $\Omega\cap I^{-}(\partial^{+}C(M))$ is contained in the domain
\[
	I^{-}_{\delta}(\partial^{+}C(M))=\{ p \in I^{-}(\partial^{+}C(M)) \ | \ d_{AdS}(p,\partial^{+}C(M))\leq \frac{\pi}{2}-3\delta\} \ ,
\]
where $d_{AdS}$ denotes the causal distance induced by the metric on $M$. A contradiction will follow if we prove that we can find $H<0$ such that the surface $S_{H}$ is at time-like distance bigger than $\frac{\pi}{2}-3\delta$ from the upper-boundary of the convex core. Since
\[
	\lim_{K\to -1}d(K)=\frac{\pi}{2}\ ,
\]
we can find a surface $\Sigma_{K_{\epsilon}}$ with constant Gauss curvature $-1-\epsilon$ such that
\[
	d(-1-\epsilon)\geq \frac{\pi}{2}-\delta\ ,
\]
and the surface $\Sigma_{K_{\epsilon}}$ lies in a $\delta$-neighbourhood of the upper boundary of the convex core. This last property is guaranteed by Theorem \ref{KfoliationAdS}: when the curvature tends to $-1$, the constant curvature surface converges uniformly to the boundary of the convex core of $M$. Now, by Lemma \ref{keylemmaAdS}, the equidistant surface $S_{\epsilon}$ at time-like distance $d(-1-\epsilon)$ from $\Sigma_{K_{\epsilon}}$ is orthogonal to the singular lines, has constant mean curvature $f(-1-\epsilon)$ and belongs to the family $\{S_{H}\}_{H \in \R}$. On the other hand, the time-like distance between $S_{\epsilon}$ and the upper-boundary of the convex core is bigger than $\pi/2-2\delta$ by construction. \\
\indent In a similar way we can exclude that the family $\{S_{H}\}_{H\in \R}$ does not foliate the entire future of the maximal surface of $M$. If this were the case, there would exist $\epsilon>0$ such that the upper boundary of $\Omega$ would be at time-like distance less than $\frac{\pi}{2}-3\epsilon$ from the lower-boundary of the convex core. Since
\[
	\lim_{K\to-\infty}d(K)=0\ ,
\]
then there exists $K_{0}<0$ such that for every $K<K_{0}$, the constant mean curvature surface 
$S_{f(K)}$ is at time-like distance less than $\epsilon$ from the correspondent $K$-surface $\Sigma_{K}$. By Theorem \ref{KfoliationAdS}, we can find a point $p\in M$ such that
\[
	d_{AdS}(p, \partial^{-}C(M))\geq \frac{\pi}{2}-\epsilon
\]
and such that $p \in \Sigma_{K}$ for some $K<K_{0}$. But now, the point $\psi^{-d(K)}(p) \in S_{f(K)}$ obtained from $p$ by following the normal flow is at time-like distance greater than $\frac{\pi}{2}-2\epsilon$ from the lower-boundary of the convex core. But this contradicts the fact that $\psi^{-d(K)}(p) \in \Omega$.

Note that $\Sigma_{K}$ is disjoint from any other $\Sigma_{K'}$ for all $K'\not=K\in(-\infty,-1)$ (see Theorem \ref{KfoliationAdS}). Combined with the construction of $S_H$, it follows that $S_H$ is disjoint from any other $S_{H'}$ for all $H\not=H'\in\R$.

Combining all the results above, the proposition follows.

\end{proof}

We then deduce that this foliation is unique by an application of the Maximum Principle.
\begin{prop}\label{prop:uniqueCMC} The foliation by constant mean curvature surfaces $\{S_{H}\}_{H\in \R}$ is unique.
\end{prop}
\begin{proof}It is sufficient to prove that, if $S$ is a space-like surface, orthogonal to the singular locus, with constant mean curvature $H$, then $S$ must coincide with the surface $S_{H}$ in the foliation.
We can define a continuous function $F:M \rightarrow \R$ such that for every $H \in \R$, the level set $F^{-1}(H)$ is the $H$-surface $S_{H}$ in the foliation provided by the previous proposition.
Since $S$ is compact, the function $F$ admits a minimum $H_{min}$ and a maximum $H_{max}$ on $S$. By construction the surface $S_{H_{min}}$ is tangent to $S$ at a point $p_{min}$, the surface $S_{H_{max}}$ is tangent to $S$ at a point $p_{max}$, $S_{H_{min}}$ is in the past of $S$ and $S_{H_{max}}$ is in the future of $S$. We have to distinguish three cases:
\begin{enumerate}
	\item $p_{min}$ and $p_{max}$ are both non-singular points. In this case, by the classical Maximum Principle (\cite{BBZAdSCMC}), the mean curvature of $S$  is not bigger than the mean curvature of $S_{H_{min}}$ at the point $p_{min}$ and the mean curvature of $S$ is not smaller than the mean curvature of $S_{H_{max}}$ at the point $p_{max}$. Since $S$ has constant mean curvature $H$ we deduce that
\[
	H_{max}\leq H \leq H_{min} \leq H_{max} \ .
\]
Hence, $F$ is constant on $S$ and $S$ coincides with the level set $F^{-1}(H)$, as claimed.
	\item $p_{min}$ and $p_{max}$ are both singular points. Since $S$, $S_{H_{min}}$ and $S_{H_{max}}$ are all umbilical at the singular points, we can apply again the Maximum Principle (Lemma 3.12 \cite{QiyuAdS}) and deduce that
\[
	H_{max}\leq H \leq H_{min} \leq H_{max} \ .
\]
The conclusion then follows as before.
	\item if only one between $p_{min}$ and $p_{max}$ is a singular point, we can conclude by combining the two previous arguments.
\end{enumerate}
\end{proof}

\subsection{Application to landslides} In \cite{landslide1} the authors introduced an $S^{1}$-action on the product of two copies of the Teichm\"uller space of a closed, oriented surface of genus $\tau\geq 2$. This construction was later extended by \cite{QiyuAdS} to the case of surfaces with cone singularities of angles less than $\pi$. In the original definition, this landslide flow had an interpretation in terms of constant Gauss curvature surfaces embedded in CGHM AdS manifolds (possibly with particles). Then, the second author proved in \cite{TambuCMC} that landslides can also be induced by constant mean curvature surfaces. The proof presented there can be easily adapted in this singular context.

\begin{defi}\label{def:landslides}
An orientation-preserving diffeomorphism $f:(\Sigma_{\mathpzc{p}}, h)\rightarrow (\Sigma_{\mathpzc{p}},h')$ between hyperbolic surfaces with cone singularities is an \emph{$\alpha$-landslide} if there exists a conformal structure $c$ on $\Sigma_{\mathpzc{p}}$ (called the \emph{center}) and two harmonic maps isotopic to the identity (fixing the marked points) $g:(\Sigma_{\mathpzc{p}},c)\rightarrow (\Sigma_{\mathpzc{p}}, h)$ and $g':(\Sigma_{\mathpzc{p}},c)\rightarrow (\Sigma_{\mathpzc{p}}, h')$ such that
\[
	f=g'\circ g^{-1} \ \ \ \text{and} \ \ \ \Hopf(g')=e^{2i\alpha}\Hopf(g)\ .
\]
\end{defi}

\indent Landslides are related to the aforementioned $S^{1}$-action on $\dTp$ in the following way. Given a couple $(h,h')\in \dTp$, there exists a unique conformal class $c$ on $\Sigma_{\mathpzc{p}}$ such that the unique harmonic maps $g:(\Sigma_{\mathpzc{p}}, c)\rightarrow (\Sigma_{\mathpzc{p}},h)$ and $g':(\Sigma_{\mathpzc{p}}, c)\rightarrow (\Sigma_{\mathpzc{p}},h')$ isotopic to the identity have opposite Hopf differentials. For every $e^{i\alpha}\in S^{1}$ we can thus define
\begin{align*}
	L_{e^{i\alpha}}:\dTp &\rightarrow \dTp\\
		(h,h') &\mapsto (h_{\alpha}, h_{-\alpha})
\end{align*}
where $(h_{\alpha},h_{-\alpha}) \in \dTp$ are characterised by the fact that the unique harmonic maps isotopic to the identity $g_{\alpha}:(\Sigma_{\mathpzc{p}},c)\rightarrow (\Sigma_{\mathpzc{p}}, h_{\alpha})$ and $g_{-\alpha}:(\Sigma_{\mathpzc{p}},c)\rightarrow (\Sigma_{\mathpzc{p}}, h_{-\alpha})$ satisfy
\[
	\Hopf(g_{\pm\alpha})=e^{\pm i\alpha}\Hopf(g) \ ,
\]
i.e. $g_{\alpha}\circ g_{-\alpha}^{-1}$ is an $\alpha$-landslide with the center $c$.\\

A natural way to construct diffeomorphisms between hyperbolic surfaces is provided by AdS geometry (see for instance \cite{bon_schl}, \cite{BonSepKsurfacesAdS}, \cite{TambuCMC} for more details). Let $S\subset M$ be a space-like surface embedded into a CGHM AdS manifolds with particles, orthogonal to the singular lines. We can construct two hyperbolic metrics on $\Sigma_{\mathpzc{p}}$ from the embedding data of $S$:
\[
	h_{l}=I((E+JB)\cdot, (E+JB)\cdot) \ \ \ \ h_{r}=I((E-JB)\cdot, (E-JB)\cdot)
\]
where $J$ is the unique complex structure compatible with the induced metric on $S$ and $E:TS \rightarrow TS$ is the identity operator. It was verified (\cite{Schlenker-Krasnov}) that $h_{l}$ and $h_{r}$ carry 
cone singularities of the same angle $\theta_{i}$ at the points $p_{i}$ as the surface $S$. Let us denote by
\[
	\pi_{l,r}:(\Sigma_{\mathpzc{p}}, I) \rightarrow (\Sigma_{\mathpzc{p}}, h_{l,r})
\]
the identity maps. In \cite{landslide1} the authors proved that if $S$ is a space-like surface with constant Gauss curvature $-\frac{1}{\cos^{2}(\alpha/2)}$, then $\pi_{l}\circ \pi_{r}^{-1}$ is an $\alpha$-landslide. This result was then extended to the case with conical singularities in \cite{QiyuAdS}. We prove that landslides between hyperbolic metrics with cone singularities (of angles less than $\pi$) can also be constructed in terms of constant mean curvature surfaces embedded in CGHM AdS manifolds with particles.

\begin{prop}Let $S \subset M$ be a space-like surface embedded in a CGHM AdS manifold orthogonal to the singular lines. If $S$ has constant mean curvature $H$, then the map $\pi_{l}\circ \pi_{r}^{-1}$ is an $\alpha$-landslide, where
\[
	\alpha=-\arctan(H)+\frac{\pi}{2} \ .
\]
\end{prop}
\begin{proof} It is sufficient to prove that $\pi_{l,r}:(\Sigma_{\mathpzc{p}}, I) \rightarrow (\Sigma_{\mathpzc{p}}, h_{l,r})$ are harmonic with
\[
	\Hopf(\pi_{l})=e^{2i\alpha}\Hopf(\pi_{r}) \ .
\]
In order to prove that $\pi_{r}$ is harmonic, it is sufficient to prove that if we write
\[
	\pi_{r}^{*}h_{r}=I(\cdot, b\cdot) \ ,
\]
then the traceless part of $b$ is Codazzi for $I$. By definition of the map $\pi_{r}$, we know that
\[
	\pi_{r}^{*}h_{r}=I((E-JB)\cdot, (E-JB)\cdot)=I(\cdot, (E-JB)^{*}(E-JB)) \ ,
\]
where $(E-JB)^{*}$ denotes the adjoint operator of $(E-JB)$ for the metric $I$. Since $B$ is $I$-self-adjoint and $J$ is skew-symmetric for $I$, we deduce that $(E-JB)^{*}=(E+BJ)$, thus $b=(E+BJ)(E-JB)$. Let us now decompose the operator $B$ as $B=B_{0}+HE$, where $B_{0}$ is traceless. Then
\begin{align*}
	(E+BJ)(E-JB)&=E+BJ-JB+B^{2}\\
		&=E+B_{0}J+HJ-JB_{0}-HJ+B_{0}^{2}+H^{2}E+2HB_{0}\\
		&=(1+H^{2})E+2(HE-J)B_{0}+B_{0}^{2}\\
		&=(1+\lambda^{2}+H^{2})E+2(HE-J)B_{0}
\end{align*}
where, in the last line, we have used the fact that $B_{0}^{2}=\lambda^{2}E$, $\lambda$ being the positive eigenvalue of $B_{0}$. We deduce that the traceless part $b_{0}$ of $b$ is given by $b_{0}=2(HE-J)B_{0}$, which is Codazzi because $H$ is constant, $B$ is Codazzi and $J$ is integrable and compatible with the metric $I$ (hence $d^{\nabla^{I}}J=0$). Therefore, $Ib_{0}$ is the real part of a meromorphic quadratic differential $2\Phi_{r}$ and the metric $\pi_{r}^{*}h_{r}$ can be written as
\[
	\pi_{r}^{*}h_{r}=I(\cdot, b\cdot)=e_{r}I+\Phi_{r}+\overline{\Phi_{r}}
\]
for some function $e_r:\Sigma_{\mathpzc{p}}\rightarrow \R^{+}$, which is the energy density of the map $\pi_{r}$. Since $\Phi_{r}$ is holomorphic for the complex structure $J$, this shows that $\pi_{r}$ is harmonic with Hopf differential
\[
	\Hopf(\pi_{r})=\Phi_{r}=\frac{1}{2}\{I(HE-J)B_{0}+iJI(HE-J)B_{0}\} \ .
\]
We claim that
\[
	\frac{\Hopf(\pi_{l})}{\Hopf(\pi_{r})}=\frac{H+i}{H-i}\ .
\]
With the result in \cite{TambuCMC}, it suffices to check near the cone singularities. In the conformal coordinate $z$ near $p_i$ for $I$, we have $I=e^{2u_i}|z|^{2\beta_i}|dz|^{2}$ and
\[
	\Hopf(\pi_{r})=e^{2u_i}|z|^{2\beta_i}(H-i)\lambda^2 dz^{2} \ .
\]
A similar computation shows that $\pi_{l}$ is harmonic with Hopf differential
\[
	\Hopf(\pi_{l})=e^{2u_i}|z|^{2\beta_i}(H+i)\lambda^2 dz^{2} \ .
\]
We deduce that
\[
	\frac{\Hopf(\pi_{l})}{\Hopf(\pi_{r})}=\frac{H+i}{H-i}=e^{2i\alpha}
\]
with
\[
	\alpha=-\arctan(H)+\frac{\pi}{2}
\]
and that $\pi_{l}\circ \pi_{r}^{-1}$ is an $\alpha$-landslide.
\end{proof}

\section{The de-Sitter case}\label{sec:dS}
The main result of this section is the following:

\begin{teo}\label{mainteodS}Let $M$ be a CGHM de Sitter manifold with particles. For every $H\in (-\infty, -1)$, there exists a unique space-like surface $S_{H}$ embedded in $M$, orthogonal to the singular locus and with constant mean curvature $H$. Moreover, the surfaces $(S_{H})_{H \in(-\infty,-1)}$ provide a foliation of the manifold $M$ (with $H$ varying from $-\infty$ near the initial singularity to $-1$ near the boundary at infinity).
\end{teo}

To prove this, we will use the following recent result proved by the first author in a joint work with Jean-Marc Schlenker:

\begin{teo}[\cite{QiyudS}]\label{KfoliationdS} Let $M$ be a CGHM de Sitter manifold with particles. For every $K\in (-\infty, 0)$ there exists a unique future-convex space-like surface $\Sigma_{K}$ orthogonal to the singular locus and with constant Gauss curvature $K$. Moreover, the surfaces $(\Sigma_{K})_{K \in (-\infty, 0)}$ provide a $C^2$-foliation of $M$ outside the singular locus (with $K$ varying from $-\infty$ near the initial singularity to $0$ near the boundary at infinity).
\end{teo}

The main idea of the proof of Theorem \ref{mainteodS} is similar to that for AdS case. We will verify that there exist two functions $\delta:(-\infty,0) \rightarrow \R^{+}$ and $\xi:(-\infty, 0) \rightarrow \R$ such that, given a $K$-surface $\Sigma_{K}$, the surface equidistant to $\Sigma_{K}$ in the future of $\Sigma_{K}$ at time-like distance $\delta(K)$ has constant mean curvature $\xi(K)$. Moreover, we will show that $\xi(K)$ takes value over $(-\infty,-1)$ as $K$ takes value over $(-\infty,0)$. This implies that we find a surface $S_H$ of constant mean curvature $H$ in $M$ for every $H<-1$.  Then we deduce that the family $(S_{H})_{H\in(-\infty,-1)}$ must foliate the whole manifold $M$ by Theorem \ref{KfoliationdS}.\\

Now we recall the behavior of the principal curvatures of the equidistant space-like surfaces in a CGHM de Sitter manifold with particles along the normal flow:

\begin{lemma}\label{principalcurvaturesdS} Let $S$ be a future-convex space-like surface embedded into a CGHM de Sitter manifold $M$ with particles. Let $\psi^{t}:S \rightarrow M$ be the flow along the future-directed unit normal vector field $\eta$ on $S$. For every regular point $x \in S$, we have
\begin{enumerate}
	\item the induced metric and the shape operator on $\psi^{t}(S)$ are
\begin{equation*}
\begin{split}
	I_{t}&=I\big((\cosh (t)E-\sinh (t)B)\cdot, (\cosh (t)E-\sinh (t)B)\cdot\big),\\
	B_{t}&=\big(\cosh (t)E-\sinh (t)B\big)^{-1}\big(\cosh (t) B-\sinh (t) E\big),
\end{split}
\end{equation*}	
where $I$ and $B$ are the induced metric and the shape operator on $S$.
	\item the principal curvatures of $\psi^{t}(S)$ at the point $\psi^{t}(x)$ are given by
\[
	\lambda_{t}(\psi^{t}(x))=\frac{\lambda(x)-\tanh (t)}{1-\tanh (t)\lambda(x)} \ \ \ \ \ \ \ \ \
	\mu_{t}(\psi^{t}(x))=\frac{\mu(x)-\tanh (t)}{1-\tanh (t)\mu(x)} \ ;
\]
	\item $\psi^{t}$ is an embedding in a neighbourhood of $x$ if $t$ satisfies $\tanh (t)\lambda(x)\neq 1$ and $\tanh (t)\mu(x)\neq 1$, where $\lambda(x)$ and $\mu(x)$ are the principal curvatures of $S$ at $x$.

\item If $\lambda(x)\mu(x)>1$, then the function $F(t)=\lambda_{t}(\psi^{t}(x))\mu_{t}(\psi^{t}(x))$ is strictly decreasing in $(0, +\infty)$.
\end{enumerate}
\end{lemma}
\begin{proof} Since Statements (2), (3) and (4) follow by a direct computation from Statement (1), we just prove (1). It is sufficient to prove it for $M=dS_{\theta}$. Let $\sigma:S \rightarrow M$ be the embedding. The geodesic orthogonal to $S$ at a point $x=\sigma(y)$ can be written as
\[
	\gamma_{x}(t)=\cosh (t)\sigma(y)+\sinh (t)\eta(\sigma(y)) \ .
\]
Therefore,
\begin{align*}
	I_{t}(v,w)&=\langle d\gamma_{x}(t)(v),d\gamma_{x}(t)(w)\rangle_{3,1} \\
		&=\langle \cosh (t) d\sigma_{x}(v)+\sinh (t) d(\eta \circ \sigma)_{x}(v), \cosh (t) d\sigma_{x}(w)+\sinh (t) d(\eta \circ \sigma)_{x}(w) \rangle_{3,1}\\
		&=\sigma_{x}^{*}g_{dS_{\theta}}(\cosh (t) v-\sinh (t)Bv,\cosh (t) w-\sinh (t)Bw)\\
		&=I((\cosh (t)E-\sinh (t)B)v, (\cosh (t) E-\sinh (t)B)w) \ .
\end{align*}
Note that $\II_{t}=-\frac{1}{2}\frac{dI_{t}}{dt}$, hence
\[
	\II_{t}=-\frac{1}{2}\frac{dI_{t}}{dt}=I\big((\sinh (t))E-\cosh (t)B)\cdot, \sinh(t)B-\cosh (t)E)\cdot\big),
\]
and
\[
	B_{t}=I_{t}^{-1}\II_{t}=\big(\cosh (t)E-\sinh (t)B\big)^{-1}\big(\cosh (t)B-\sinh (t)E\big)  \ .
\]
\end{proof}
\begin{remark}\label{rk: dS singular points}
The formulas in Lemma \ref{principalcurvaturesdS} hold true also at the singular points, whenever the principal curvatures at those points are well-defined.
\end{remark}

Note that with our convention a future-convex space-like surface $S$ has always negative principal curvatures, thus the equidistant surfaces $\psi^{t}(S)$ obtained by moving along the normal flow in the future of $S$ are well-defined for every $t>0$.

\begin{lemma}\label{keylemmadS}Let $\Sigma_{K}$ be a future-convex space-like surface with constant Gauss curvature $K<0$, orthogonal to the singular lines. Then the surface $\psi^{\delta(K)}(\Sigma_{K})$ in the future of $\Sigma_{K}$ has constant mean curvature
\[
	\xi(K)=\frac{K-2}{2\sqrt{1-K}} \ \ \ \ \ \ \ \ \ \text{if} \ \ \ \ \ \ \ \ \ 
	\delta(K)=\arctanh \left(\sqrt{\frac{1}{1-K}}\right)
\]
\end{lemma}
\begin{proof}Denote by $\lambda$ and $\mu$ the principal curvatures of the surface $\Sigma_{K}$. By assumption (and by the Gauss equation) the product $\kappa=\lambda\mu$ is constant. Denote by $\lambda_{t}$ and $\mu_{t}$ the principal curvatures of the surface $\psi^{t}(\Sigma_{K})$. We want to show that there exists a time $t>0$, such that the mean curvature $H_{t}=\frac{\mu_{t}+\lambda_{t}}{2}$ is constant on $\psi^t(\Sigma_K)$.
To prove this, it suffices to find $t>0$ such that $dH_{t}=0$. By Lemma \ref{principalcurvaturesdS}, we have
\begin{align*}
	2dH_{t}=(d\lambda_{t}+d\mu_{t})= \frac{(1-\tanh^2(t))d\lambda}{(1-\tanh(t)\lambda)^2} +\frac{(1-\tanh^2(t))d\mu}{(1-\tanh(t)\mu)^2} \ .
\end{align*}
This vanishes identically if and only if
\begin{equation}\label{eq:differential of H_t for dS}
\begin{split}
0&=d\mu(1-\tanh^{2}(t))(1-\lambda\tanh(t))^{2}
+d\lambda(1-\tanh^{2}(t))(1-\mu\tanh(t))^{2}\\ &=(1-\tanh^{2}(t))[d\mu(1+\lambda^{2}\tanh^{2}(t))
+d\lambda(1+\mu^{2}\tanh^{2}(t))]\\
&-2(1-\tanh^{2}(t))\tanh(t)(\lambda d\mu+\mu d\lambda) \ .
\end{split}
\end{equation}

By substituting $\mu=\frac{\kappa}{\lambda}$ in \eqref{eq:differential of H_t for dS} and using the fact that $0=d(\kappa)=d(\lambda\mu)=\mu d\lambda+\lambda d\mu$, we get
\begin{align*}
	0&=\frac{1}{\lambda^{2}}(1-\tanh^{2}(t))
(-\kappa d\lambda+\lambda^{2}d\lambda
-\kappa\tanh^{2}(t)\lambda^{2}d\lambda
+\kappa^{2}\tanh^{2}(t)d\lambda)\\	&=\frac{d\lambda}{\lambda^{2}}(1-\tanh^{2}(t))(\lambda^{2}-\kappa)(1-\kappa\tanh^{2}(t))\ .
\end{align*}
Let $\delta(K)$ denote the time function $t$ with respect to $K$ for which $1-\kappa\tanh^{2}(t)=0$. Its expression follows from the Gauss equation that $\kappa=1-K$.
The formula for the mean curvature follows then by a direct computation using the Lemma \ref{principalcurvaturesdS}:
\begin{equation*}
\begin{split}
	2H_{\delta(K)}&=\lambda_{\delta(K)}+\mu_{\delta(K)}
=\frac{\lambda-\sqrt{\frac{1}{\lambda\mu}}}{1+\sqrt{\frac{\lambda}{\mu}}}
+\frac{\mu-\sqrt{\frac{1}{\lambda\mu}}}{1+\sqrt{\frac{\mu}{\lambda}}}\\
	&=\frac{\lambda+\frac{1}{\lambda}-2\sqrt{\lambda\mu}
-2\sqrt{\frac{1}{\lambda\mu}}+\mu+\frac{1}{\mu}}
{2+\sqrt{\frac{\lambda}{\mu}}+\sqrt{\frac{\mu}{\lambda}}}
=\frac{\lambda+\mu+\frac{1}{\mu}
+\frac{1}{\lambda}
-2\sqrt{\kappa}
-\frac{2}{\sqrt{\kappa}}}
{2-\frac{\sqrt{\kappa}}{\mu}-\frac{\sqrt{\kappa}}{\lambda}}\\
	&=\frac{\kappa^{2}+\mu^{2}\kappa+\kappa+\mu^{2}-2\kappa\sqrt{\kappa}\mu-2\sqrt{\kappa}\mu}
{2\kappa\mu-\kappa\sqrt{\kappa}-\mu^{2}\sqrt{\kappa}}\\
&=-\frac{(\kappa+1)(\mu^{2}+\kappa-2\sqrt{\kappa}\mu)}
{\sqrt{\kappa}(\mu^{2}+\kappa-2\sqrt{\kappa}\mu)} =\frac{-\kappa-1}{\sqrt{\kappa}}
=\frac{K-2}{\sqrt{1-K}}\ ,
\end{split}
\end{equation*}
where we used the fact that $\lambda\mu=\kappa$ and in the last equality we used the Gauss formula $\kappa=1-K$.
\end{proof}

Therefore, we can define two continuous functions
\[
	\delta:(-\infty,0)\rightarrow \R^{+} \ \ \ \ \ \ \text{and} \ \ \ \ \ \ \xi:(-\infty, 0) \rightarrow \R
\]
with the property that, given a future-convex space-like surface $\Sigma_{K}$ with constant Gauss curvature $K<0$ and orthogonal to the singular locus, the equidistant surface $S_{H}$ in the future of $\Sigma_{K}$ at  time-like distance $\delta(K)$ has constant mean curvature $H=\xi(K)$. Note that the function $\xi$ takes over $(-\infty, -1)$ as $K$ takes over $(-\infty,0)$, combined with  Theorem \ref{KfoliationdS}, we have the following:

\begin{cor}Given a CGHM de Sitter manifold with particles, for every $H<-1$ there exists a space-like surface $S_{H}$ orthogonal to the singular locus, with constant mean curvature $H$.
\end{cor}

 Observe that the surface $S_{H}$ obtained in this way is umbilical (i.e. the principal curvatures extend continuously and coincide) at the singular points. Indeed, the $K$-surface $\Sigma_{K}$, provided by Theorem \ref{KfoliationdS}, is umbilical at the singular points, which follows from the dual result in \cite[Theorem 1.5]{QiyuAdS}. See also \cite[Section 6.2]{QiyuAdS}. 
Moreover, $S_{H}$ is obtained by pushing $\Sigma_{K}$ along the normal flow, hence the claim follows from Lemma \ref{principalcurvaturesdS} and Remark \ref{rk: dS singular points}. This property actually holds for every space-like surface $S$ of constant mean curvature, orthogonal to the singular lines (see the following lemma, which follows from the same argument as Lemma \ref{lm:AdSmeancurvature} for the AdS case). This will be crucial in order to apply the Maximum Principle in this singular context.

\begin{lemma}Let $S$ be a space-like surface, orthogonal to the singular locus, with constant mean curvature $H$, embedded in a CGHM de Sitter manifold with particles. Then the principal curvatures of $S$ at the singular points coincide and are equal to $H$.
\end{lemma}

Let $S_H$ denote the surface in a CGHM de Sitter manifold with particles $M$ obtained from the future-convex space-like surface $\Sigma_K$ with constant curvature $K<0$ with $\xi(K)=H$ as constructed in Lemma \ref{keylemmadS}. We have the following proposition.

\begin{prop}\label{existencefoliationdS}The surfaces $\{S_{H}\}_{H\in(-\infty,-1)}$ foliate the whole manifold $M$.
\end{prop}
\begin{proof}
Let $\Omega$ be the union of the constant mean curvature surfaces $S_{H}$ over all $H\in(-\infty,-1)$.
The same argument as that of Proposition \ref{prop:HfoliationMink} shows that if $H_{n}$ is a sequence converging to $H \in (-\infty, +\infty)$, then the surface $S_{H_{n}}$ converges to $S_{H}$ uniformly everywhere and smoothly outside the singular locus.\\
\indent We now prove that the region $\Omega$ must coincide with the whole manifold $M$. Recall that if $K_n\rightarrow 0$, the time-like distance from the surface $\Sigma_{K_n}$ to the initial singularity of $M$ tends to infinity (this is a dual description for the de Sitter case of Statement (2) in \cite[Lemma 5.16]{QiyudS}). Combined with the construction of $S_{H_n}$ from $\Sigma_{K_n}$, we have $\cup_n I^{-}(S_{H_n})=M$. In particular, $H_n\rightarrow -1$ as $K_n\rightarrow 0$.

It remains to show that if $H_n\rightarrow -\infty$, then $\cup_n I^{+}(S_{H_n})=M$. Otherwise, there exists a point, say $x$, in $M$ which is not contained in the future of surfaces $S_{H_n}$ for any $n$. Let $D$ be the maximal distance from $x$ to the initial singularity of $M$. Theorem \ref{KfoliationdS} implies that there exists a surface $\Sigma_{K_N}$ of constant Gauss curvature $K_N$ for $N$ large enough such that $\delta(K_N)<D/2$ (this is ensured by the fact that $\delta(K_n)=\arctanh \left(\sqrt{\frac{1}{1-K_n}}\right)$ tends to $0$ as $K_n$ tends to $-\infty$), lying within a time-like distance less than $D/2$ from the initial singularity of $M$. Therefore, the equidistant surface $S_{H_N}$ in the future at a distance $\delta(K_N)$ from $\Sigma_{K_N}$, lies within a time-like distance strictly less than $D$ from the initial singularity. This implies that $x$ is in the future of $S_{H_N}$, which leads to contradiction.

Note that $\Sigma_{K}$ is disjoint from any other $\Sigma_{K'}$ for all $K'\not=K\in(-\infty,0)$ (see Theorem \ref{KfoliationdS}). Combined with the construction of $S_H$, it follows that $S_H$ is disjoint from any other $S_{H'}$ for all $H\not=H'\in(-\infty,-1)$.

Combining all the results above, the proposition follows.

\end{proof}

Note that the argument using the Maximum Principle in Proposition \ref{prop:uniqueCMC} for the AdS case is also adapted to the de Sitter case. We have the following result.

\begin{prop}\label{prop:uniqueness of CMC for dS}The foliation of a CGHM de Sitter manifold with particles by constant mean curvature surfaces is unique.
\end{prop}

Combined with Proposition \ref{existencefoliationdS} and Proposition \ref{prop:uniqueness of CMC for dS}, we complete the proof of Theorem \ref{mainteodS}.

\bibliographystyle{alpha}
\bibliography{bs-bibliography}

\end{document}